\documentclass[]{amsart}

\usepackage{amssymb}
\usepackage{amsfonts}
\usepackage{amsmath}
\usepackage{amsthm}
\usepackage{graphicx}
\usepackage{mathrsfs}
\usepackage{dsfont}
\usepackage{amscd}
\usepackage{multirow}
\usepackage[all]{xy}
\normalfont
\usepackage[T1]{fontenc}
\usepackage{calligra}
\usepackage{verbatim}
\usepackage[usenames,dvipsnames]{color}
\usepackage[colorlinks=true,linkcolor=Blue,citecolor=Violet]{hyperref}
\usepackage{enumerate}

\newcommand{\N}{{\mathds{N}}}

\newcommand{\R}{{\mathds{R}}}
\newcommand{\C}{{\mathds{C}}}
\newcommand{\T}{{\mathds{T}}}

\newcommand{\D}{{\mathfrak{D}}}
\newcommand{\A}{{\mathfrak{A}}}
\newcommand{\B}{{\mathfrak{B}}}
\newcommand{\M}{{\mathfrak{M}}}
\newcommand{\F}{{\mathfrak{F}}}

\newcommand{\bigslant}[2]{{\raisebox{.2em}{$#1$}\left/\raisebox{-.2em}{$#2$}\right.}}

\newcommand{\Lip}{{\mathsf{L}}}

\newcommand{\dist}{{\mathsf{dist}}}

\newcommand{\dpropinquity}[1]{{\mathsf{\Lambda}^\ast_{#1}}}

\newcommand{\Kantorovich}[1]{{\mathsf{mk}_{#1}}}

\newcommand{\Haus}[1]{{\mathsf{Haus}_{#1}}}

\newcommand{\StateSpace}{{\mathscr{S}}}
\newcommand{\MongeKant}{{Mon\-ge-Kan\-to\-ro\-vich metric}}
\newcommand{\qcms}{quantum compact metric space}

\newcommand{\ouqcms}{quantum metric order unit space}

\newcommand{\unit}{1}
\newcommand{\sa}[1]{{\mathfrak{sa}\left({#1}\right)}}

\newcommand{\alg}[1]{{\mathfrak{#1}}}
\newcommand{\dom}[1]{{\operatorname*{dom}\left({#1}\right)}}

\newcommand{\norm}[2]{\left\|{#1}\right\|_{#2}}

\newcommand{\opnorm}[3]{{\left|\mkern-1.5mu\left|\mkern-1.5mu\left| {#1} \right|\mkern-1.5mu\right|\mkern-1.5mu\right|_{#3}^{#2}}}

\newcommand{\tunnelextent}[1]{{\chi\left({#1}\right)}}
\newcommand{\BaireSpace}{{\mathscr{N}}}
\newcommand{\CondExp}[2]{{\mathds{E}_{#1}\left({#2}\right)}}
\newcommand{\CP}[2]{{\alg{CP}\left({#1},{#2}\right)}}
\newcommand{\BunceDeddens}[1]{{\alg{BD}\left({#1}\right)}}

\theoremstyle{plain}
\newtheorem{theorem}{Theorem}[section]

\newtheorem{corollary}[theorem]{Corollary}

\newtheorem{lemma}[theorem]{Lemma}
\newtheorem{proposition}[theorem]{Proposition}

\newtheorem{theorem-definition}[theorem]{Theorem-Definition}

\theoremstyle{definition}
\newtheorem{definition}[theorem]{Definition}

\newtheorem{hypothesis}[theorem]{Hypothesis}

\theoremstyle{remark}

\newtheorem{remark}[theorem]{Remark}

\newtheorem{notation}[theorem]{Notation}

\renewcommand{\geq}{\geqslant}
\renewcommand{\leq}{\leqslant}

\numberwithin{equation}{section}

\allowdisplaybreaks[4]

\begin{document}

\title[Bunce-Deddens algebras as metric limits of circle algebras]{Bunce-Deddens algebras as quantum Gromov-Hausorff distance limits of circle algebras}
\author{Konrad Aguilar}
\email{aguilar@imada.sdu.dk}
\urladdr{}
\address{Department of Mathematics and Computer Science (IMADA) \\ University of Southern Denmark \\  Campusvej 55, DK-5230 Odense M, Denmark}
\thanks{The first author gratefully acknowledges the financial support  from the Independent Research Fund Denmark  through the project `Classical and Quantum Distances' (grant no.~9040-00107B)}
\thanks{The first and second author were partially supported by the grant H2020-MSCA-RISE-2015-691246-QUANTUM DY-NAMICS and the Polish Ministry of Science and Higher Education grant \#3542/H2020/2016/2.}
\thanks{Some of this work was completed at{\em The Fields Institute for
Research in Mathematical Sciences} during the{\em Workshop on New Geometry of Quantum Dynamics
August 12 - 16, 2019}} 

\author{Fr\'{e}d\'{e}ric Latr\'{e}moli\`{e}re}
\email{frederic@math.du.edu}
\urladdr{http://www.math.du.edu/\symbol{126}frederic}
\address{Department of Mathematics \\ University of Denver \\ Denver CO 80208}

\author{Timothy Rainone}
\email{trainone@oxy.edu }
\urladdr{https://www.oxy.edu/academics/faculty/timothy-rainone}
\address{ Department of Mathematics\\Occidental College\\ Los Angeles CA 90041}

\date{\today}
\subjclass[2000]{Primary:  46L89, 46L30, 58B34.}
\keywords{Noncommutative metric geometry, Gromov-Hausdorff convergence, Monge-Kantorovich distance, Quantum Metric Spaces, Lip-norms, Bunce-Deddens algebras, AT-algebras.}

\begin{abstract}
  We show that Bunce-Deddens algebras, which are A$\T$-algebras, are also limits of circle algebras for Rieffel's quantum Gromov-Hausdorff distance, and moreover, form a continuous family indexed by the Baire space. To this end, we endow Bunce-Deddens algebras with a quantum metric structure, a step which requires that we reconcile the constructions of the Latr\'{e}moli\`{e}re's Gromov-Hausdorff propinquity and Rieffel's quantum Gromov-Hausdorff distance when working on order-unit quantum metric spaces. This work thus continues the study of the connection between inductive limits and metric limits.
\end{abstract}
\maketitle

%%%%%%%%%%%%%%%%%%%%%%%%%%%%%%%%%%%%%%%%%%%

\section{Introduction}

Noncommutative analogues of the Gromov-Hausdorff distance \cite{Rieffel00, Latremoliere13,Latremoliere13b,Latremoliere14,Latremoliere15} allow for the discussion of limits for certain sequences unital C*-algebras endowed with a notion of a quantum metric, in a manner which generalizes the notion of convergence of compact metric spaces in the sense of the Edwards-Gromov-Hausdorff distance \cite{Gromov}. Of course, in C*-algebra theory, a very common notion of limit is given by inductive limits of inductive sequences, as used to great success in classification theory, among others. In particular, inductive limits of finite dimensional C*-algebras, called AF-algebras, can be seen as the beginning of the research on the classification of C*-algebras. Reconciling metric convergence and inductive limits for AF algebras has been the topic of a previous work from the first two authors \cite{Aguilar-Latremoliere15}, followed by several developments by the second author \cite{Aguilar16c, Aguilar18}. A next, crucial chapter in the theory of classification, was the study of inductive limits of circle algebras, called A$\T$-algebras. A well-known example of an A$\T$-algebra is given by the Bunce-Deddens algebras \cite{Bunce75}, and the present work proposes to see how metric convergence can be reconciled with the notion of inductive limit for this particular family of A$\T$-algebras. Moreover, we also prove that the function which maps an element of the Baire space --- a sequence of natural numbers --- to its associated Bunce-Deddens algebra is a continuous map for Rieffel's quantum Gromov-Hausdorff distance.

The first noncommutative analogue of the Gromov-Hausdorff distance, motivated by questions from mathematical physics, was discovered by Rieffel \cite{Rieffel00}. It was constructed over the class of \emph{\ouqcms s}, which Rieffel called  quantum compact metric spaces --- though as we shall briefly discuss, this last term's meaning has evolved in time. The idea behind the definition of an {\ouqcms}, inspired by an idea of Connes \cite{Connes89}, is to generalize the structure of a space of Lipschitz maps, endowed with the Lipschitz seminorm. Rieffel especially noted that a key property of the Lipschitz seminorm on the space of real-valued functions over a compact metric space $X$ is that it induces by duality a distance on the space of Radon probability measures over $X$ which metrizes the weak* topology. This distance is of course the {\MongeKant} introduced by Kantorovich \cite{Kantorovich40}. This property can be made sense of in the more general, noncommutative context.

\begin{definition}[{\cite{Alfsen71}}]
  A \emph{vector space with an order unit} is a pair $(V,\unit_V)$ of an ordered vector space $(V,\leq)$ over $\R$ such that for all $v \in V$ there exists $\lambda > 0$ such that $-\lambda \unit_V \leq v \leq \lambda \unit_V$.

  An \emph{order unit space} $\A$ is a vector space with an order unit $\unit_\A$ which is Archimedean, i.e. for all $v \in V$, if $v \leq \lambda \unit_\A$ for all $\lambda>0$ then $v \leq 0$.
  
  An order unit space comes equipped with a norm defined for all $a \in \A$ by $\|a\|_\A = \inf \{\lambda >0 \mid -\lambda 1_\A \leq a \leq \lambda 1_\A\}$ and  satisfying $-\|a\|_\A 1_\A \leq a \leq \|a\|_\A1_\A $.
\end{definition}

\begin{definition}
  The state space $\StateSpace(\A)$ of an order unit space $\A$ is the set of all positive linear functionals from $\A$ to $\C$ which maps the order unit of $\A$ to $1$.
\end{definition}

The first occurrence of the term \emph{quantum compact metric space} can be found in Connes' work on spectral triples \cite{Connes89}. Rieffel proposed a definition for this term based on order unit spaces and seminorms which generalize Lipschitz seminorms on spaces of functions over compact metric spaces. The following definition encapsulates Rieffel's notion as used in his construction of the quantum Gromov-Hausdorff distance.

\begin{definition}[{\cite{Rieffel98a,Rieffel99,Rieffel00}}]\label{ouqcms-def}
  A \emph{\ouqcms} $(\A,\Lip)$ is an ordered pair of a norm-complete order unit space $\A$ and a seminorm $\Lip$ defined on a norm-dense subspace of $\A$ such that:
  \begin{enumerate}
  \item $\{ a \in \A : \Lip(a) = 0 \} = \R\unit_\A$,
  \item the {\MongeKant}, defined for any two $\varphi, \psi \in \StateSpace(\A)$ by:
    \begin{equation*}
      \Kantorovich{\Lip}(\varphi,\psi) = \sup\left\{ |\varphi(a) - \psi(a)| : a\in \A, \Lip(a) \leq 1 \right\} \text{, }
    \end{equation*}
    metrizes the weak* topology on $\StateSpace(\A)$,
  \item $\{a\in\A: \Lip(a) \leq 1\}$ is closed in $\|\cdot\|_\A$.
  \end{enumerate}
  The seminorm $\Lip$ is called a Lip-norm on $\A$.
  
  We denote the class of {\ouqcms} by  $\mathrm{CQMS}_\mathrm{ou}$.
\end{definition}

We note that the requirement that the unit ball of Lip-norms be closed is not included in \cite{Rieffel00}. However, as explained in \cite{Rieffel00}, if a seminorm $S$ satisfies (1) and (2) but not (3) in Definition (\ref{ouqcms-def}), then setting:
\begin{equation*}
  \Lip(a) = \sup\left\{ \frac{|\varphi(a) - \psi(a)|}{\Kantorovich{S}(\varphi,\psi)} : \varphi \not= \psi \in \StateSpace(\A) \right\}
\end{equation*}
allowing for $\infty$, gives a {\ouqcms} $(\A,\Lip)$ with $\Kantorovich{\Lip} = \Kantorovich{S}$ and $\Lip$ is lower semicontinuous on $\A$. So Assumption (3) can always be made, and it simplifies the statement of many theorems.

Rieffel's \emph{quantum Gromov-Hausdorff distance} is thus a noncommutative analogue of the Gromov-Hausdorff distance for {\ouqcms s}. Its definition follows Edwards and Gromov's ideas, though of course, the techniques needed to establish the properties of this new metric are quite different.

\begin{definition}
  Let $(\A_1,\Lip_1)$ and $(\A_2,\Lip_2)$ be two {\ouqcms s}. A Lip-norm $\Lip$ on $\A_1\oplus\A_2$ is \emph{admissible} for $\Lip_1$ and $\Lip_2$ when:
  \begin{equation*}
    \forall j \in \{1,2\} \quad \forall a \in \A_j \quad \Lip_j(a) = \inf\left\{ \Lip(a,b) : b \in \B \right\}\text{.}
  \end{equation*}
\end{definition}

\begin{notation}
  If $(E,d)$ is a metric space, the \emph{Hausdorff distance} \cite{Hausdorff} induced by $d$ on the set of all closed subsets of $E$ is denoted by $\Haus{d}$.
\end{notation}

\begin{definition}
  The \emph{quantum Gromov-Hausdorff distance} $\dist_q((\A_1,\Lip_1),(\A_2,\Lip_2))$ between two {\ouqcms}s $(\A_1,\Lip_1)$ and $(\A_2,\Lip_2)$ is:
  \begin{equation*}
    \inf\left\{ \Haus{\Kantorovich{\Lip}}(\StateSpace(\A_1),\StateSpace(\A_2)) : \text{$\Lip$ is admissible for $\Lip_1$ and $\Lip_2$} \right\}\text{.}
  \end{equation*}
\end{definition}

Rieffel proved in \cite{Rieffel00} that the distance $\dist_q$ is a complete pseudo-metric on the class of all {\ouqcms}s, which is zero between two {\ouqcms s} $(\A,\Lip_\A)$ and $(\B,\Lip_\B)$ if and only there exists a positive linear map $\pi : \A\rightarrow\B$ such that $\Lip_\B\circ\pi = \Lip_\A$. Several examples of convergence for this metric were derived \cite{Latremoliere05,Rieffel00,Rieffel01}.

In time, it has become apparent that progress in noncommutative metric geometry requires a noncommutative analogue of the Gromov-Hausdorff distance for the class of {\qcms s} defined, not on order unit spaces, but on actual C*-algebras, with the appropriate coincidence property. The Gromov-Hausdorff propinquity \cite{Latremoliere13b,Latremoliere14} provides such an analogue. 
\begin{definition}[{\cite{Latremoliere15}}]\label{d:qcms} A {\em {\qcms} with the $F$-Leibniz property}  (where $F : [0,\infty)^4\rightarrow[0,\infty)$ is a function which is increasing in the product order) is given by an ordered pair $(\A,\Lip)$ where:
\begin{enumerate}
\item $\A$ is a unital C*-algebra,
\item $\Lip$ is a seminorm defined on a dense domain of $\sa{\A}$, where $\sa{\A}$ is the self-adjoint part of $\A$, 
\item $(\sa{\A},\Lip)$ is an {\ouqcms},
\item $\max\left\{ \Lip\left(\frac{ab+ba}{2}\right),\Lip\left(\frac{ab-ba}{2i}\right) \right\} \leq F(\norm{a}{\A},\norm{b}{\A},\Lip(a),\Lip(b))$ for all $a,b \in \sa{\A}$.
\end{enumerate}
\end{definition}
For any fixed function $F$ as above, the \emph{propinquity}, denoted $\dpropinquity{F}$, is a complete metric on the class of {\qcms s} with the $F$-Leibniz property, with the property that the propinquity between two such {\qcms s} $(\A,\Lip_\A)$ and $(\B,\Lip_\B)$ is null if and only if there exists a *-isomorphism $\pi: \A\rightarrow\B$ such that $\Lip_\B\circ\pi = \Lip_\A$, and $\pi$ is called a {\em full quantum isometry} \cite{Latremoliere15}. We note that the propinquity is a metric up to full quantum isometry on the class of all $F$-Leibniz {\qcms s} without fixed $F$, which we denote $\dpropinquity{}$, but it is not complete in this case. The propinquity, when restricted to ``classical'' quantum metric spaces, is topologically equivalent to the Gromov-Hausdorff distance. New techniques are needed to prove the properties of the propinquity and utilize it since working with {\qcms s} rather than {\ouqcms s} means working with a more rigid structure. On the other hand, the advantages of working with the propinquity become apparent as it allows for the discussion of convergence of modules \cite{Latremoliere16} or convergence of group actions \cite{Latremoliere18}.

Many convergence results are known for the propinquity \cite{Latremoliere13b,Latremoliere16,Aguilar-Latremoliere15,Rieffel15}. In particular, \cite{Aguilar-Latremoliere15}, the two first authors initiated the study of convergence, in the metric sense, of sequences used to construct C*-algebras by taking inductive limits, by studying the metric properties of AF-algebras. This work was continued in subsequent papers \cite{Aguilar16c}. As part of this particular line of investigation, the first author proved in \cite{Aguilar18} that the completeness of the propinquity makes it possible to define a quantum metric on the inductive limit of a sequence of C*-algebras endowed with quantum metric structures, as long as the connecting maps satisfy some natural properties.

The present work proposes to endow Bunce-Deddens algebras with quantum metrics using the same method as \cite{Aguilar18}. Thus, the idea is to introduce certain quantum metrics on circle algebras and a completeness-based argument to obtain a metric on the A$\T$-algebra obtained as the inductive limit of these circle algebras.

There is however a difficulty in proceeding directly along these lines. Indeed, each quantum metric which we introduce on circle algebras does satisfy a form of Leibniz identity, but we can not prove that there is a uniform choice of such a Leibniz property for the entire inductive sequence for a given Bunce-Deddens algebra. This means that unfortunately, we work outside of any class where we know that the propinquity is a metric (it is of course a pseudo-metric).

Thus, in order to use the techniques of \cite{Aguilar18}, we would like to understand how some of the relevant constructions in \cite{Latremoliere13,Latremoliere13c} for the propinquity may remain valid without the Leibniz property assumptions. As we noted, there is no hope to keep the important coincidence property, but this is not directly used in \cite{Aguilar18}. However, the proof of completeness is central to the argument of \cite{Aguilar18}.

We see in this paper that in fact, once we remove the constraints to work with C*-algebras and quasi-Leibniz Lip-norms, the construction of the propinquity, if mimicked, simply gives an alternate expression for Rieffel's distance $\dist_q$. This is a very interesting fact, since it shows that indeed, the efforts placed in devising new techniques when working with the propinquity are exactly due to working with C*-algebras and  Lip-norms with some Leibniz property. We stress that this \emph{does not} mean that the restriction of $\dist_q$ to {\qcms s} is the propinquity --- it is not as it still does not enjoy the appropriate coincidence property. What it means is that by allowing order unit spaces in the construction of the propinquity, we lose what makes the propinquity different from $\dist_q$. This observation is of independent interest.

Therefore, in this paper, we prove that Bunce-Deddens algebras are limits of their inductive sequences in a metric sense, for Rieffel's distance $\dist_q$. There is one more point of subtlety which we must address here. While Rieffel proved that $\dist_q$ is complete in \cite{Rieffel00}, we need for our proof in this paper a different description of the limit of a Cauchy sequence for $\dist_q$. The description we seek is essentially the one obtained in \cite{Latremoliere13b} for the propinquity. Thus, we spend some efforts carrying the proof of completeness in \cite{Latremoliere13b} to $\dist_q$. This turns out to be quite technical, but is carried out successfully in this paper. This is an example of how our new expression for $\dist_q$, inspired by the propinquity, can lead to new observations about $\dist_q$.

Our paper is thus organized as follows. First, in the next section, we open with a general scheme to turn bi-Lipschitz morphisms between C*-algebras to quantum isometries by manipulating the quantum metrics on their codomain, under the assumption that the range of the morphisms is also the range of a conditional expectation on the codomain. This section is very general, and is concluded with a theorem about making Cauchy sequences for $\dist_q$ out of inductive limits of {\qcms s}. We already observe that in general, this construction suffers from the problem that no quasi-Leibniz property emerges which is common to all the quantum metrics in the Cauchy sequence, forcing us to work with $\dist_q$ rather than the propinquity. This second section contains the main idea of the construction of quantum metrics on circle algebras used in this paper.

We then apply our second section to the standard inductive sequences defining Bunce-Deddens algebras in the third and fourth section of this paper. We establish that all the needed ingredients required to apply the first section can be constructed for these circle algebras. We conclude with the fact that such an inductive sequence is naturally Cauchy for $\dist_q$, and thus by completeness, must converge to some {\ouqcms}. The question, of course, is whether this {\ouqcms} is the self-adjoint part of a Bunce-Deddens algebra with some quantum metric.

To answer this question, we generalize \cite{Aguilar18} to $\dist_q$ under appropriate hypothesis. In the fifth section, we establish a new expression for $\dist_q$ inspired, as discussed above, by the propinquity. We then see how the proof of completeness for the propinquity gives a new proof of completeness for $\dist_q$ which, importantly, gives a different expression for the limit of a Cauchy sequence. Of course, our expression for these limits are isomorphic as {\ouqcms s} to Rieffel's, but our new expression makes it possible to relate the limit in the metric sense to the limit in the categorical sense. This matter is explained in the sixth section of the paper, where \cite{Aguilar18} is ported to our current framework. We then can answer our problem and prove that, indeed, Bunce-Deddens algebras are limits for $\dist_q$ of their standard inductive sequence. We use this result to also obtain a continuity result for the family of Bunce-Deddens algebras over the Baire space.

\section{Distance from Conditional Expectations and bi-Lipschitz monomorphisms}

We begin with a very simple observation: if a C*-subalgebra of a {\qcms} $(\A,\Lip)$, containing the unit of $\A$, is the range of some conditional expectation on $\A$ which is also contractive for $\Lip$, then we can always modify $\Lip$ to make $\A$ and $\B$ arbitrarily close in the sense of the propinquity (though never at distance $0$ unless $\A$ and $\B$ are *-isomorphic, of course).

During this section, we will keep track of the Leibniz conditions on our quantum metrics, precisely because in fact, it will make clear the difficulties we encountered with our construction about inductive limits. For this purpose, we set $F_{2,0} : (x,y,l_x,l_y) \in [0,\infty)^4 \mapsto 2 ( x l_y + y l_x )$.

\begin{lemma}\label{lemma:cond-lip}
  Let $(\A,\Lip)$ be an $F$-Leibniz {\qcms} and let $\B\subseteq\A$ be a C*-subalgebra of $\A$ which contains the unit of $\B$. If there exists a conditional expectation $\mathds{E} : \A \twoheadrightarrow \B$ of $\A$ onto $\B$ such that $\Lip\circ\mathds{E} \leq \Lip$, then for all $\varepsilon > 0$, if we set for all $a\in\sa{\A}$:
  \begin{equation*}
    \Lip_\varepsilon(a) = \max\left\{ \Lip(a), \frac{1}{\varepsilon}\norm{ a - \mathds{E}(a) }{\A} \right\}
  \end{equation*}
  then $(\A,\Lip_\varepsilon)$ is a $\max\{F,F_{2,0}\}$-Leibniz {\qcms}, $\Lip_\varepsilon(b) = \Lip(b)$ for all $b\in \sa{\B}$, and:
  \begin{equation*}
    \dpropinquity{\max\{F,F_{2,0}\}}((\A,\Lip_\varepsilon),(\B,\Lip)) \leq \varepsilon \text{.}
  \end{equation*}
\end{lemma}

\begin{proof}
  The domain of $\Lip_\varepsilon$ is the domain of $\Lip$ since $N : a\in\A \mapsto \frac{1}{\varepsilon}\norm{a - \mathds{E}(a)}{\A}$ is continuous. Moreover, this also implies that $\Lip_\varepsilon$ is lower semi-continuous seminorm. Moreover, $N$ is $(2,0)$-quasi-Leibniz by \cite[Lemma 2.3]{Aguilar-Latremoliere15}, so $\Lip_\varepsilon$ is $\max\{F,F_{2,0}\}$-Leibniz. Of course, if $\Lip_\varepsilon(a) = 0$ then $\Lip(a) = 0$ so $a\in\R\unit_\A$, and $\Lip_\varepsilon(\unit_\A) =0$.

  Last, for any $\mu\in\StateSpace(\A)$:
  \begin{equation*}
    \left\{ a \in \sa{\A} : \Lip_\varepsilon(a) \leq 1, \mu(a) = 0 \right\} \subseteq \left\{ a \in \sa{\A} : \Lip(a) \leq 1, \mu(a) = 0 \right\}
  \end{equation*}
and thus, as the set on the right hand side is totally bounded since $\Lip$ is an L-seminorm, so it the set on the left hand side. We have thus shown that $(\A,\Lip_\varepsilon)$ is a $\max\{F,F_{2,0}\}$-Leibniz {\qcms} using \cite[Proposition 1.3]{Ozawa05}.

By assumption, $\mathds{E}(b) = b$ for all $b\in \B$ so $\Lip_\varepsilon(b) = \Lip(b)$ for all $b\in \sa{\B}$.

Let $\pi_\A : (a,b)\in\A\oplus\B\mapsto a \in \A$ and $\pi_\B : (a,b)\in\A\oplus\B\mapsto b \in \B$.

  For all $a\in\sa{\A}$ and $b\in\sa{\B}$, we define:
  \begin{equation*}
    \mathsf{Q}(a,b) = \max\left\{ \Lip_\varepsilon(a), \Lip(b), \frac{1}{\varepsilon} \norm{b - a}{\A}  \right\}\text{.}
  \end{equation*}
  
  If $b\in\sa{\B}$ and $\Lip(b) = 1$ then $\Lip_\varepsilon(b) = 1$ and $\norm{b-b}{\A} = 0$. Thus $\mathsf{Q}(b,b) = 1$. Since $\mathds{Q}(a,b) \geq \Lip(b)$ for all $a\in\sa{\A}$, we conclude that $\pi_\B : \A\oplus\B \mapsto \B$ is a quantum isometry.

  If $a \in \sa{\A}$ and $\Lip_\varepsilon(a) = 1$ then $\Lip(\mathds{E}(a)) \leq 1$, and moreover:
  \begin{equation*}
    \norm{ a - \mathds{E}(a) }{\A} \leq \varepsilon
  \end{equation*}
  so $\mathsf{Q}(a,\mathds{E}(a)) = 1$. Again, since $\mathsf{Q}(a,b)\geq\Lip_\varepsilon(a)$ for all $b\in\sa{\B}$, we conclude that  $\pi_\B$ is a quantum isometry as well.

  We thus gather that $(\A\oplus\B,\mathsf{Q},\pi_\A,\pi_\B)$ is a tunnel from $(\A,\Lip_\varepsilon)$ to $(\B,\Lip)$ by \cite[Definition 2.3]{Latremoliere14}. It is of course $\max\{F,F_{2,0}\}$-Leibniz. We now compute its extent.

  Let $\varphi \in \StateSpace(\B)$. Let $\psi = \varphi\circ\mathds{E} \in \StateSpace(\A)$ and note that $\varphi = \varphi\circ\mathds{E}$, so the restriction of $\psi$ to $\B$ is $\varphi$ --- in fact for this proof, any extension $\psi$ of $\varphi$ to a state of $\A$ would work (but since we have our conditional expectation here, we need not invoke the Hahn-Banach theorem). Let $(a,b)\in\sa{\A\oplus\B}$ with $\mathsf{Q}(a,b)\leq 1$. Then in particular, $\norm{a-b}{\A} \leq \varepsilon$ and thus:
  \begin{equation*}
    \left|\psi(a) - \varphi(b)\right| = \left|\psi(a-b)\right| \leq \varepsilon \text{.}
  \end{equation*}
  Therefore $\Kantorovich{\mathsf{Q}}(\varphi,\psi) \leq \varepsilon$.

  Let now $\varphi \in \StateSpace(\A)$. Set $\psi$ be the restriction of $\varphi$ to $\B$. Let $(a,b) \in \sa{\A\oplus\B}$ with $\mathsf{Q}(a,b)\leq 1$. Then $\norm{a - b}{\A} \leq 1$ and thus:
  \begin{equation*}
    \left|\varphi(b)-\psi(a)\right| = \left|\varphi(b-a)\right| \leq \varepsilon\text{.}
  \end{equation*}
  We conclude $\Kantorovich{\mathsf{Q}}(\varphi,\psi) \leq\varepsilon$.

  Consequently $\tunnelextent{\tau} \leq \varepsilon$. This concludes our proof by \cite[Definition 3.6]{Latremoliere14}.
\end{proof}

We now turn to a result about bi-Lipschitz *-morphisms. There are several equivalent definitions of Lipschitz morphisms \cite{Rieffel99,Latremoliere16b}, which we now recall.

\begin{definition}
  A \emph{Lipschitz morphism} $\varphi : (\A,\Lip_\A) \rightarrow (\B,\Lip_\B)$ between two {\qcms s} $(\A,\Lip_\A)$ and $(\B,\Lip_\B)$ is a unital *-morphism from $\A$ to $\B$ such that:
  \begin{equation*}
    \varphi(\dom{\Lip_\A}) \subseteq \dom{\Lip_\B}\text{.}
  \end{equation*}
\end{definition}

\begin{theorem}[{\cite{Latremoliere16b}}]\label{t:lip-morph}
  Let $\varphi : \A\rightarrow\B$ be a unital *-morphism between two unital C*-algebras $\A$ and $\B$. If $\Lip_\A$ and $\Lip_\B$ are, respectively, Lip-norms on $\A$ and $\B$, the the following assertions are equivalent:
  \begin{enumerate}
  \item $\varphi$ is a Lipschitz morphism from $(\A,\Lip_\A)$ to $(\B,\Lip_\B)$,
  \item $\exists k \geq 0 \quad \Lip_\B\circ\varphi \leq k \Lip_\A$,
  \item $\exists k \geq 0 \quad \forall\mu,\nu \in \StateSpace(\B) \quad \Kantorovich{\Lip_\A}(\mu\circ\varphi,\nu\circ\varphi) \leq k \Kantorovich{\Lip_\B}(\mu,\nu)$.
  \end{enumerate}
\end{theorem}

If $\alpha : (\A,\Lip_\A) \rightarrow (\B,\Lip_\B)$ is a $k$-Lipschitz morphism between two {\qcms s}, then it naturally becomes a contractive morphism from $(\A,\Lip_\A)$ to $(\B, \frac{1}{k} \Lip_\B)$. Yet if $\alpha$ is actually injective and bi-Lipschitz, then it is usually not possible to adjust the quantum metrics to turn $\pi$ into a quantum isometry. However, under the hypothesis that we can find a conditional expectation $\mathds{E}$ of $\B$ onto the range of $\alpha$ in $\B$ such that $\mathds{E}$ is also a Lipschitz linear map, then it is indeed possible to modify $\Lip_\B$ to turn $\alpha$ into a quantum isometry.

\begin{lemma}\label{lemma:iso}
  Let $(\A,\Lip_\A)$ and $(\B,\Lip_\B)$ be {\qcms s}, with $\Lip_\B$ an $F$-Leibniz seminorm. If $\alpha : \A\rightarrow \B$ is a unital *-monomorphism and $\mathds{E} : \B \twoheadrightarrow \alpha(\A)$ is a conditional expectation such that for some $m_\alpha, k_{\mathds{E}} > 0$, the following conditions hold:
  \begin{itemize}
  \item $\forall a \in \sa{\A} \quad m_\alpha \Lip_\A(a) \leq \Lip_\B\circ\alpha(a)$,
  \item $\forall a \in \sa{\B} \quad \Lip_\B(\mathds{E}(a)) \leq k_{\mathds{E}}\Lip_\B(a)$,
  \end{itemize}
  
  then if we set:

  \begin{equation*}
    \forall b \in \sa{\B} \quad \Lip'_\B(b) = \max\left\{ \Lip_\B(b), \Lip_\A(\alpha^{-1}\circ\mathds{E}(a)) \right\}
  \end{equation*}
  where $\alpha^{-1}$ is the inverse of the *-isomorphism $\alpha : \A\rightarrow\alpha(\A)$, then $(\A,\Lip'_\A)$ is a $\max\left\{1, \frac{k_{\mathds{E}}}{m_{\alpha}}\right\} F$-Leibniz {\qcms} and:
  \begin{equation*}
    \forall a \in \alpha(\A) \quad \Lip_\A(a) \leq \Lip'_\B(\alpha(a)) \text{,}
  \end{equation*}
  and $\dom{\Lip'_\B} = \dom{\Lip_\B}$.

  If, moreover, there exists $k_\alpha > 0$ such that $\Lip_\B\circ\alpha \leq k_\alpha \Lip_\A$ then:
  \begin{equation*}
    \forall a\in\sa{\A} \quad \Lip_\A(a) \leq  \Lip'_\B(\alpha(a)) \leq \max\{1,k_\alpha\} \Lip_\A(a)
  \end{equation*}
  In particular, if $k_\alpha = 1$ then $\alpha : (\A,\Lip_\A) \rightarrow (\alpha(\A),\Lip'_\B)$ is a full quantum isometry.
\end{lemma}

\begin{proof}

  Let $b\in\dom{\Lip_\B}$. By assumption, $\Lip_\B(\mathds{E}(b)) \leq k_{\mathds{E}} \Lip_\B(b) <  \infty$ and therefore $\Lip_\A(\alpha^{-1}(\mathds{E}(b))) \leq \frac{1}{m_\alpha} \Lip_\B(\mathds{E}(b)) \leq \frac{k_{\mathds{E}}}{m_\alpha} \Lip_\B(b) < \infty$. So $b \in \dom{\Lip'_\B}$. Therefore $\dom{\Lip_\B}\subseteq\dom{\Lip'_\B}$ and since $\dom{\Lip_\B}$ is dense in $\sa{\B}$, so is $\dom{\Lip'_\B}$. Moreover, if $\Lip_\B(a) = \infty$ then $\Lip'_\B(a) = \infty$, so $\dom{\Lip'_\B}\subseteq\dom{\Lip_\B}$. Thus $\dom{\Lip_\B} = \dom{\Lip'_\B}$.

  If $\Lip'_\B(b) = 0$ then $\Lip_\B(b) = 0$ so $b \in \R\unit_\B$. Of course, $\Lip'_\B(\unit_\B) = 0$ since $\mathds{E}$ and $\alpha$ are unital.

  We now check the Leibniz property of $\Lip'_\B$. For all $a,b \in \sa{\B}$, we compute:
  \begin{align*}
    \Lip_\A(\alpha^{-1}(\mathds{E}(a b)))
    &\leq \frac{1}{m_\alpha} \Lip_\B (\mathds{E}(a b)) \\
    &\leq \frac{k_{\mathds{E}}}{m_\alpha} \Lip_\B(a b)\\
    &\leq \frac{k_{\mathds{E}}}{m_\alpha} F\left(\Lip_\B(a), \Lip_\B(b), \norm{a}{\B}, \norm{b}{\B}\right) \text{.}
  \end{align*}
  Therefore, $\Lip'_\B$ is $\max\left\{1, \frac{k_{\mathds{E}}}{m_{\alpha}}\right\} F$-Leibniz.
  
  As the supremum of two lower semi-continuous seminorms, $\Lip'_\B$ is a lower semi-continuous seminorm as well. Moreover, if $\mu \in\StateSpace(\B)$, since:
  \begin{equation*}
    \left\{ b \in \sa{\B} : \Lip'_\B(b)\leq 1, \mu(b) = 0 \right\} \\
    \subseteq \left\{ b \in \sa{\B} : \Lip_\B(b) \leq 1, \mu(b) = 0 \right\} 
  \end{equation*}
  and since $\Lip_\B$ is an Lip-norm, we conclude that the set on the right hand side, and therefore the set on the left hand side, is totally bounded.

  Therefore $\Lip'_\B$ is a $\max\left\{1, \frac{k_{\mathds{E}}}{m_{\alpha}}\right\} F$-quasi-Leibniz Lip-norm as a claimed.

  By construction of $\Lip'_\B$:
  \begin{equation*}
    \Lip_\A(a) = \Lip_\A(\alpha^{-1}(\alpha(a))) = \Lip_\A(\alpha^{-1}(\mathds{E}(\alpha(a)))) \leq \Lip'_\B(\alpha(a)) \text{.}
  \end{equation*}

  Last, assume that there exists $k_\alpha > 0$ such that $\Lip_\B\circ\alpha \leq k_\alpha \Lip_\A$. We will note that, of course, if $a\in\A$ then $\alpha^{-1}(\mathds{E}(\alpha(a))) = a$, and thus:
  \begin{equation*}
    \Lip'_\B(\alpha(a)) = \max\{ \Lip_\B(\alpha(a)), \Lip_\A(\alpha^{-1}\mathds{E}(\alpha(a))) \} \leq \max\{ k_\alpha, 1\} \Lip_\A(a)\text{.}
  \end{equation*}
 
  If $k_\alpha = 1$ then $\Lip_\A(a) = \Lip'_\B(\alpha(a))$ for all $a\in\alpha(\A)$. Thus $\alpha$ is a full quantum isometry form $(\A,\Lip_\A)$ onto $(\alpha(\A),\Lip'_\B)$.
\end{proof}

%\begin{remark}In particular, if $\Lip_\B$ is $(C,D)$-quasi-Leibniz then $\Lip'_\B$ is $(\max\{C, \frac{k_{\mathds{E}}}{m_\alpha} C\},\ D)$-quasi-Leibniz.\end{remark}

\begin{remark}
  If we drop the assumption that $\Lip_\B$ is $F$-Leibniz, then we still can conclude that $\Lip'_\B$ is a lower semi-continuous Lip-norm such that $\alpha^{-1}$ is $1$-Lipschitz.
\end{remark}

We now bring our two previous observations in one theorem which will be key to our construction.

\begin{theorem}\label{theorem:bilip}
  Let $(\A,\Lip_\A)$ and $(\B,\Lip_\B)$ be {\qcms s}, with $\Lip_\B$ a $F$-Leibniz seminorm. If $\alpha : \A\rightarrow \B$ is a unital *-monomorphism and $\mathds{E} : \B \twoheadrightarrow \alpha(\A)$ is a conditional expectation such that for some $m_\alpha > 0$:
  \begin{itemize}
  \item $\forall a \in \sa{\A} \quad m_\alpha \Lip_\A(a) \leq \Lip_\B\circ\alpha(a) \leq \Lip(a)$,
  \item $\forall a \in \sa{\B} \quad \Lip_\B(\mathds{E}(a)) \leq \Lip_\B(a)$,
  \end{itemize}
  and if $\varepsilon > 0$, then setting:
  \begin{equation*}
    \forall b \in \sa{\B} \quad \Lip^{\varepsilon}_\B(b) = \max\left\{ \Lip_\B(b), \frac{1}{\varepsilon}\norm{ b - \mathds{E}(b) }{\B}, \Lip_\A(\alpha^{-1}\circ\mathds{E}(a)) \right\}
  \end{equation*}
  then
  \begin{equation*}
    \dpropinquity{G}((\A,\Lip_\A),(\B,\Lip'_\B)) \leq \varepsilon \text{,}
  \end{equation*}
  where $G = \max \{1,\frac{1}{m_\alpha}\} \cdot F$.
\end{theorem}

\begin{proof}

  By Lemma (\ref{lemma:iso}), if we define $\Lip'_\B = \max\{ \Lip_\B, \Lip_\A\circ\alpha^{-1}\circ\mathds{E} \}$ then $\alpha$ is a full quantum isometry from $(\A,\Lip_\A)$ to $(\alpha(\A),\Lip'_\B)$. So:
  \begin{equation*}
    \dpropinquity{G}((\A,\Lip_\A),(\alpha(\A),\Lip'_\B)) = 0\text{.}
  \end{equation*}

  The seminorm $\Lip^{\varepsilon}_\B$ is then obtained by applying Lemma (\ref{lemma:cond-lip}) to $\Lip'_\B$. In particular, $\Lip^\varepsilon_\B(b) = \Lip'_\B(b)$ for all $b\in \sa{\alpha(\A)}$, and:
  \begin{equation*}
    \dpropinquity{G}((\alpha(\A),\Lip'_\B), (\B,\Lip^\varepsilon_\B)) \leq \varepsilon \text{.}
  \end{equation*}

  Thus, by the triangle inequality, we conclude:
  \begin{equation*}
    \dpropinquity{G}((\A,\Lip_\A),(\B,\Lip^\varepsilon_\B)) \leq \varepsilon
  \end{equation*}
  as desired.
\end{proof}

We record that the tunnel constructed in Theorem (\ref{theorem:bilip}) is given as:
\begin{equation}\label{eq:tunnel}
  \mathsf{Q}(a,b) = \max\left\{ \Lip_\A(a), \Lip^\varepsilon_\B(b), \frac{1}{\varepsilon} \norm{b - \alpha(a)}{\B}  \right\}
\end{equation}
for all $(a,b) \in \sa{\A\oplus \B}.$

We have now worked out how, given a bi-Lipschitz morphism between two {\qcms s}, it is possible to change the quantum metric on the codomain of this morphism to turn it into a full quantum isometry onto its range and to make its range and its codomain arbitrarily close in the quantum propinquity, at the cost of relaxing the Leibniz inequality.

We now apply this construction repeatedly to an inductive sequence of {\qcms s} whose connecting maps are all bi-Lipschitz morphisms. The problem which arises is that unfortunately, the Leibniz condition of the seminorms we construct typically worsen at each stage. We will address this matter after we prove that we can indeed make a Cauchy sequence for $\dist_q$ out of any inductive sequence of {\qcms s} where the connecting maps are bi-Lipschitz.

\begin{theorem}\label{theorem:inductive}
  Let:
  \begin{equation*}
    \A_0 \xrightarrow{\alpha_0} \A_1 \xrightarrow{\alpha_1} \A_2 \xrightarrow{\alpha_2} \A_3 \xrightarrow{\alpha_3} \ldots
  \end{equation*}
  be an inductive sequence of unital C*-algebras, where the connecting maps are unital *-monomorphisms, such that:
  \begin{itemize}
  \item for each $n\in\N$, we are given an $F_n$-Leibniz Lip-norm $\Lip_n$ on $\sa{\A_n}$,
  \item for each $n\in\N\setminus\{0\}$, there exists a conditional expectation $\mathds{E}_{n}$ from $\A_{n}$ onto $\alpha_{n-1}(\A_{n-1})$ such that $\Lip_{n}\circ\mathds{E}_{n}\leq \Lip_n$,
  \item for each $n\in\N\setminus \{0\}$, there exists $c_n,d_n > 0$ such that:
    \begin{equation*}
      c_n \Lip_{n-1} \leq \Lip_n\circ\alpha_{n-1} \leq d_n\Lip_{n-1}
    \end{equation*}
  \end{itemize}
  then, setting $\mathsf{S}_0 = \Lip_0$ and for all $n\in\N\setminus\{0\}$:
  \begin{equation*}
    \forall a \in \sa{\A_n} \quad \mathsf{S}_n(a) = \max\left\{ \varkappa_n \Lip_n(a) , \mathsf{S}_{n-1}\circ\alpha_{n-1}^{-1}\circ\mathds{E}_n(a) , \frac{1}{2^n}\norm{a - \mathds{E}_n(a)}{\A_n}  \right\}
  \end{equation*}
  where
  \begin{equation*}
    \varkappa_n = \prod_{j=0}^n \frac{1}{d_j}
  \end{equation*}
  then:
  \begin{equation*}
   \dpropinquity{}((\A_n,\mathsf{S}_n), (\A_{n+1},\mathsf{S}_{n+1})) \leq \frac{1}{2^{n+1}} \text{.}
  \end{equation*}

  Consequently, there exists a {\ouqcms} $(O,\Lip)$ such that:
  \begin{equation*}
    \lim_{n\rightarrow\infty} \dist_q((\sa{\A_n},\mathsf{S}_n),(O,\Lip)) = 0 \text{.}
  \end{equation*}
\end{theorem}

\begin{proof}
  This proof is by induction. Since $c_0 \Lip_0 \leq \Lip_1\circ\alpha_0 \leq d_0 \Lip_0$, the map $\alpha_0$ is a contraction from $(\A_0,\Lip_0)$ to $(\A_1,\frac{1}{d_0}\Lip_1)$, and $\frac{c_0}{d_0}\Lip_0\leq \Lip_1\circ\alpha_0$. Thus by Theorem (\ref{theorem:bilip}), if we set:
  \begin{equation*}
    \forall a \in \sa{\A_1} \quad \mathsf{S}_1(a) = \max\left\{ \frac{1}{d_0}\Lip_1(a), \Lip_0\circ\alpha_0^{-1}\circ\mathds{E}_0(a), \norm{ a - \mathds{E}_0(a) }{\A_1} \right\}
  \end{equation*}
  then $(\A_1,\mathsf{S}_1)$ is an $F$-Leibniz quantum compact metric space for some $F$, and $\dpropinquity{}((\A_0,\Lip_0), (\A_1,\mathsf{S}_1))\leq 1$.
  
  Assume that, for some $n\in\N$, we have now shown that $\dpropinquity{}((\A_k,\mathds{S}_k),(\A_{k+1},\mathsf{S}_{k+1})) \leq \frac{1}{2^k}$ and $\dom{\mathsf{S}_k} = \dom{\Lip_k}$ for all $k\in\{1,\ldots,n\}$.

  By assumption
  \begin{equation*}
    \Lip_{n+1}\circ \alpha_n \leq d_{n+1} \Lip_n \leq \frac{d_{n+1}}{\varkappa_n}\mathsf{S}_n  = \frac{1}{\varkappa_{n+1}}\mathsf{S}_n \text{.}
  \end{equation*}

  Let $a\in\sa{\alpha_n(\A_n)}$ and $\varepsilon > 0$. There exists $a_\varepsilon \in \dom{\Lip_{n+1}}$ such that $\norm{a - a_\varepsilon}{\A_{n+1}} < \varepsilon$. By assumption, $\Lip_{n+1}\circ\mathds{E}_{n+1}(a_\varepsilon) \leq \Lip_{n+1}(a_\varepsilon) < \infty$, Moreover, $\mathds{E}_{n+1}(a_\varepsilon) \in \alpha_n(\A_{n+1})$ and
  \begin{align*}
    \norm{ \mathds{E}_{n+1}(a_\varepsilon) - a }{\A_{n+1}}
    &= \norm{ \mathds{E}_{n+1}(a_\varepsilon) - \mathds{E}_{n+1}(a) }{\A_{n+1}} \\
    &\leq \norm{ a_\varepsilon - a }{\A_{n+1}} < \varepsilon \text{.}
  \end{align*}
  Thus $\dom{\Lip_{n+1}} \cap \sa{\alpha_{n}(\A_{n})}$ is dense in $\sa{\alpha_{n}(\A_n)}$. Again by our assumption, if $a\in\sa{\alpha_n(\A_n)}\cap\dom{\Lip_{n+1}}$, then $\Lip_{n}(\alpha_n^{-1}(a)) \leq \frac{1}{c_n}\Lip_{n}(a) < \infty$. Thus by our induction hypothesis, we conclude $\alpha_n^{-1}(a) \in \dom{S_n}$. Now $\alpha_n^{-1}$ is a unital *-morphism from $(\alpha_n(\A_n),\Lip_{n+1})$ onto $(\A_n,\mathsf{S}_n)$ which maps $\dom{\Lip_{n+1}}$ to $\dom{\mathsf{S}_n}$. By  Theorem (\ref{t:lip-morph}), we conclude that there exists $m > 0$ such that $m \mathsf{S}_{n}\leq \Lip_{n+1}\circ\alpha_n$. We will compute an estimate for $m$ in the next lemma but its actual value is not very important for us.

  Thus, by Theorem (\ref{theorem:bilip}), if we set:
  \begin{multline*}
    \forall a \in \sa{\A_{n+1}} \quad \mathsf{S}_{n+1}(a) = \\
    \max\left\{ \varkappa_{n+1}\Lip_{n+1}(a), \mathsf{S}_n\circ\alpha_n^{-1}\circ\mathds{E}_n(a), \frac{1}{2^{n+1}}\norm{ a - \mathds{E}_{n+1}(a) }{\A_{n+1}} \right\}
  \end{multline*}
  then $\mathsf{S}_{n+1}$ is an $F$-Leibniz Lip-norm for  some $F$ and:
  \begin{equation*}
    \dpropinquity{}((\A_{n+1},\mathsf{S}_{n+1}), (\A_n,\mathsf{S}_n)) \leq \frac{1}{2^{n+1}} \text{,}
  \end{equation*}
  as claimed. Moreover $\dom{\mathsf{S}_{n+1}} = \dom{\Lip_{n+1}}$. Our induction hypothesis holds for all $n\in\N$.

  The conclusion of the theorem then follows from the observation that $\dist_q$ is complete and is dominated by $\dpropinquity{}$ by \cite[Theorem 5.5]{Latremoliere13b}.
\end{proof}

We note that the above theorem highlights the issue we have with the lack of a uniform Leibniz rule for all $n \in \N$. Indeed, we see that we can calculate distances between any term in our inductive sequence using propinquity since propinquity is a distance on the class of all quantum compact metric spaces equipped with any $F$-Leibniz property as discussed after Definition \ref{d:qcms}. However, we are unable to find a uniform bound on the $F$-Leibniz properties over all $n \in \N$ (even in the explicit setting of the Bunce-Deddens algebras) and propinquity is only complete over classes of quantum compact metric spaces with uniform $F$-Leibniz property. But, this is where $\mathrm{dist}_q$ has an advantage. It is complete on the class of all quantum metric order unit spaces. Although $\mathrm{dist}_q$ doesn't have the optimal coincidence property, it still has this one crucial advantage, and this is why the above theorem ends with $\mathrm{dist}_q$ rather than $\dpropinquity{}$.

Of course, we want to relate the limit in Theorem  (\ref{theorem:inductive}) with the inductive limit of the given sequence. This is achieved by proving two observations:
\begin{itemize}
\item the limit of the sequence is described in the proof of the completeness of $\dist_q$,
\item the tunnels used in the proof of Theorem (\ref{theorem:bilip}), which are given by Expression (\ref{eq:tunnel}), are actually at once related to the inductive limit and to the metric limit computations.
\end{itemize}
Before we move in this direction, we however begin with our core example for Theorem(\ref{theorem:inductive}): the Bunce-Deddens algebras.

\section{The Bunce-Deddens C*-algebras}

We denote the C*-algebra of $n\times n$ matrices over $\C$ by $\M_n(\C)$.

\begin{definition}[{\cite{Miller95}}]\label{d:Baire}
  Let $\BaireSpace=(\N\setminus \{0,1\})^{\N\setminus\{0\}}$. For each $x,y \in \BaireSpace$ set
  \begin{equation*}
    \mathsf{d}_{\BaireSpace}(x,y)=
    \begin{cases}
      0 \text{ if $x=y$,} \\
      2^{-\min\{m \in \N : x(m) \neq y(m)\}} \text{ otherwise.}
    \end{cases}
  \end{equation*}
  
  The metric space $(\BaireSpace, \mathsf{d}_{\BaireSpace})$ is called the {\em Baire space}.
\end{definition}

\begin{notation}
  If $\sigma \in \BaireSpace$ then, for $m\in\N\setminus\{0\}$, we set:
  \begin{equation*}
    \boxtimes\sigma_m = \prod_{j=1}^m \sigma_m \text{;}
  \end{equation*}
  and we set $\boxtimes\sigma_0 = 1$, and we denote $\boxtimes \sigma=(\boxtimes \sigma_m)_{m \in \N }$.
\end{notation}

\begin{notation}\label{notation:z}
  For each $t \in \R$, let $z(t)=\exp(2 \pi i t)$. Let $m \in \N\setminus\{0\}$.
  \begin{itemize}
  \item For all $j,k \in \{1, \ldots, m\}$ and $t\in\R$, we define
    \begin{equation*}
      z^m_{j,k}(t) = \frac{1}{\sqrt{m}}z((m-j)(t+k-1)/m) \text{.}
    \end{equation*}
    The function $z^m_{j,k}$ is a continuous, $m$-periodic $\C$-valued function over $\R$.
    
  \item We define the $\M_m(\C)$-valued continuous function:
    \begin{equation*}
      U_m = (z^m_{j,k})_{j,k \in \{1, \ldots,m\}} = \begin{pmatrix} z^m_{1,1} &\cdots & z^m_{1,m} \\ \vdots & & \vdots \\ z^m_{m,1} & \cdots & z^m_{m,m} \end{pmatrix} \text{.}
    \end{equation*}
  \end{itemize}
  In addition, we set $U_0 = 1$. The function $U_m$ is a continuous $m$-periodic $\M_m(\C)$-valued function. Moreover, by \cite[Chapter V.3]{Davidson}, the map  $U_m$ is unitary.
\end{notation}

\begin{lemma}
  For all $m\in\N\setminus\{0\}$, the map $U_m$ is a unitary and is $m$-periodic and if $V_m=\begin{pmatrix}
    0 & 0 & \cdots & 0 & 1\\
    1 & 0 & \cdots & 0 & 0 \\
    0 & 1 & \cdots & 0 & 0 \\
    \vdots & \vdots & \ddots & \vdots & 0 \\
    0 & 0 & \cdots & 1 & 0
  \end{pmatrix} \in \M_m(\C)$:
  \begin{equation*}
    \forall t \in \R \quad U_m(t+1) = U_m(t) V_m \text{.}
  \end{equation*}
\end{lemma}

\begin{proof}
  This is an immediate computation.
\end{proof}

\begin{notation}
  Let $\sigma = (\sigma_m)_{m\in\N} \in \BaireSpace$. For $m\in \N\setminus\{0\}$, the C*-algebra of $\M_{\boxtimes\sigma_m}(\C)$-valued, continuous, $1$-periodic functions over $\R$ is denoted by $\CP{\sigma}{m}$.

  We then define:
  \begin{equation*}
    U_{\sigma,m} =   U_{\sigma_{m}} \otimes \mathrm{id}_{\boxtimes\sigma_{m-1}}=
    \begin{pmatrix}
      z_{1,1}^{\sigma_m}\cdot \mathrm{id}_{\boxtimes\sigma_{m-1}} & & z_{1,\sigma_{m}}^{\sigma_m}\cdot \mathrm{id}_{\boxtimes\sigma_{m-1}} \\
      \vdots & & \vdots \\
      z_{\sigma_m,1}^{\sigma_m}\cdot \mathrm{id}_{\boxtimes\sigma_{m-1}} & & z_{\sigma_m,\sigma_m}^{\sigma_m}\cdot \mathrm{id}_{\boxtimes\sigma_{m-1}}
    \end{pmatrix}
  \end{equation*}
  and note that $U_{\sigma,m}$ is a $\M_{\boxtimes\sigma_{m}}(\C)$-valued continuous function over $\R$.
\end{notation}

\begin{lemma}\label{lemma:block-unitary}
  The map $U_{\sigma,m}$ is a unitary, $\sigma_m$-periodic function such that if $W_{\sigma,m} = V_{\sigma_m} \otimes \mathrm{id}_{\boxtimes\sigma_{m-1}}
  =\begin{pmatrix}
    0 & 0 & \cdots & 0 & \mathrm{id}_{\boxtimes\sigma_{m-1}} \\
    \mathrm{id}_{\boxtimes\sigma_{m-1}} & 0 & \cdots & 0 & 0 \\
    0 &\mathrm{id}_{\boxtimes\sigma_{m-1}} & \cdots & 0 & 0 \\
    \vdots & \vdots & \ddots & \vdots & 0 \\
    0 & 0 & \cdots & \mathrm{id}_{\boxtimes\sigma_{m-1}} & 0
  \end{pmatrix}, $
  then:
  \begin{equation*}
    \forall t\in \R \quad U_{\sigma,m} (t+1) = U_{\sigma,m}(t) W_{\sigma,m} \text{.}
  \end{equation*}
\end{lemma}

\begin{proof}
  This is an immediate computation.
\end{proof}

\begin{notation}
  Let $\sigma \in \BaireSpace$. For $m \in \N, a \in \CP{\sigma}{m}, t\in \R$, we define:
  \begin{equation*}
  \begin{split}
    &\alpha_{\sigma,m}(a)(t)\\
    & = U_{\sigma,m+1}(t) \begin{pmatrix} a\left(\frac{t}{\sigma_{m+1}}\right) & & & \\ & a\left(\frac{t+1}{\sigma_{m+1}}\right) & & \\ & & \ddots & \\ & & & a\left(\frac{t + \sigma_{m+1} - 1}{\sigma_{m+1}}\right) \end{pmatrix} U_{\sigma, m+1}^\ast (t),
    \end{split}
    \end{equation*} 
    and note $\alpha_{\sigma,m}(a)\in  C_b(\R,\M_{\boxtimes\sigma_{m+1}}(\C))\text{.}$
\end{notation}
The following lemma is presented in \cite[Section V.3]{Davidson}, but we provide more details for the proof here since our notation differs some from this reference.
\begin{lemma}\label{lemma:alpha}
  Let $\sigma\in\BaireSpace$ and $m \in \N\setminus\{0\}$. If $a\in \CP{\sigma}{m-1}$, then $\alpha_{\sigma,m-1}(a) \in \CP{\sigma}{m}$. 
  
  The map $\alpha_{\sigma,m-1}$ thus defined is a unital *-monomorphism from $\CP{\sigma}{m-1}$ to $\CP{\sigma}{m}$.
\end{lemma}

\begin{proof}
    We use the notations of Lemma (\ref{lemma:block-unitary}). If $a = ( a_{j,k} )_{1\leq j,k \leq \sigma_m} \in \M_{\boxtimes\sigma_m}(\C)$, with $a_{1,1}$, \ldots, $a_{\sigma_m,\sigma_m}$ elements of $\M_{\boxtimes\sigma_{m-1}}(\C)$, then $W_{\sigma,m} a W_{\sigma,m}^\ast$ is the matrix $( a_{\pi(j),\pi(k)} )_{1\leq j,k \leq \sigma_m}$ where $\pi$ is the permutation $(\sigma_m \ 1 \ 2 \ \ldots \ \sigma_m-1) $.
  
  By Lemma (\ref{lemma:block-unitary}), we also have:
  \begin{equation*}
    \forall t\in \R \quad U_{\sigma,m}(t+1) = U_{\sigma,m}(t) W_{\sigma,m} \text{.}
  \end{equation*}
  
  Let $a\in \CP{\sigma}{m-1}$ and $t\in\R$. Since $a$ is $1$-periodic, we note that \[a\left(\frac{t + \sigma_{m}-1+1}{\sigma_{m}}\right) = a\left(\frac{t + \sigma_{m} }{\sigma_{m}}\right) = a\left(\frac{t}{\sigma_{m}} + 1\right) =a\left(\frac{t}{\sigma_{m}} \right).\] We then compute:
  \begin{align*}
    &\alpha_{\sigma,m-1}(a)(t+1)\\
    &= U_{\sigma,m}(t+1)   \begin{pmatrix} a\left(\frac{t+1}{\sigma_{m}}\right) & & & & \\ & a\left(\frac{t+2}{\sigma_{m}}\right) & & & \\ & & \ddots & & \\ & & & a\left(\frac{t + \sigma_{m}-1}{\sigma_{m}}\right) & \\ & & & & a\left(\frac{t + \sigma_{m}-1+1}{\sigma_{m}}\right) \end{pmatrix}\\
    &\quad \quad \cdot  U_{\sigma,m}^\ast(t+1) \\
    &= U_{\sigma,m}(t+1)   \begin{pmatrix} a\left(\frac{t+1}{\sigma_{m}}\right) & & & & \\ & a\left(\frac{t+2}{\sigma_{m}}\right) & & & \\ & & \ddots & & \\ & & & a\left(\frac{t + \sigma_{m}-1}{\sigma_{m}}\right) & \\ & & & & a\left(\frac{t}{\sigma_{m}}\right) \end{pmatrix}\\
    &\quad \quad \cdot  U_{\sigma,m}^\ast(t+1) \\
      &= U_{\sigma,m}(t)W_{\sigma,m}   \begin{pmatrix} a\left(\frac{t+1}{\sigma_{m}}\right) & & & & \\ & a\left(\frac{t+2}{\sigma_{m}}\right) & & & \\ & & \ddots & & \\ & & & a\left(\frac{t + \sigma_{m}-1}{\sigma_{m}}\right) & \\ & & & & a\left(\frac{t}{\sigma_{m}}\right) \end{pmatrix}\\
    &\quad \quad \cdot  W_{\sigma,m}^\ast U_{\sigma,m}^\ast(t) \\
      &= U_{\sigma,m}(t) \begin{pmatrix} a\left(\frac{t}{\sigma_{m}}\right) & & & \\ & a\left(\frac{t+1}{\sigma_{m}}\right) & & \\ & & \ddots & \\ & & & a\left(\frac{t + \sigma_{m}-1}{\sigma_{m}}\right) \end{pmatrix} U^\ast_{\sigma,m} (t)\\
    &= \alpha_{\sigma,m-1}(a) (t) \text{.}
  \end{align*}

  Thus $\alpha_{\sigma,m-1}(a)$ is $1$-periodic. It is of course a continuous function over $\R$ valued in $\M_{\boxtimes\sigma_m}(\C)$, so $\alpha_{\sigma,m-1}(a)\in\CP{\sigma}{m}$. Since $U_{\sigma,m}$ is a unitary, it is immediate that $\alpha_{\sigma,m-1}$ is a unital *-monomorphism.
\end{proof}

\begin{definition}[{\cite[Section V.3]{Davidson}}]\label{d:bd}
  The \emph{Bunce-Deddens algebra $\BunceDeddens{\sigma}$} is the C*-algebra inductive limit \cite[Section 6.1]{Murphy90} of the sequence:
  \begin{equation*}
    \BunceDeddens{\sigma} = \underrightarrow{\lim} \ \left( \CP{\sigma}{m}, \alpha_{\sigma,m} \right)_{m\in\N}\text{.}
  \end{equation*}
\end{definition}

\begin{remark}
  The Bunce-Deddens algebra $\BunceDeddens{\sigma}$ is denoted by $\B(\boxtimes\sigma)$ in \cite[Section V.3]{Davidson}.
\end{remark}

\begin{notation}
Let $\sigma\in\BaireSpace$. For each $m \in \N$, we  let
\begin{equation*}
  \alpha^{(m)}_\sigma : \CP{\sigma}{m} \longrightarrow \BunceDeddens{\sigma}
\end{equation*}
denote the canonical unital *-monomorphism such that $\alpha^{(m+1)}_\sigma \circ \alpha_{\sigma,m} = \alpha^{(m)}_\sigma$ given by \cite[Section 6.1]{Murphy90}. 
\end{notation}

The Bunce-Deddens algebras have a unique faithful tracial state.

\begin{notation}
  Let $\sigma\in\BaireSpace$. We denote the unique faithful tracial state on $\BunceDeddens{\sigma}$ by $\tau_\sigma$ \cite[Theorem V.3.6]{Davidson}. For each $m \in \N\setminus\{0\}$, we denote
  \begin{equation*}
    \tau_{\sigma,m} = \tau_\sigma \circ \alpha^{(m)}_\sigma :  \CP{\sigma}{m} \longrightarrow \C \text{,}
  \end{equation*}
  which is a faithful tracial state on $\CP{\sigma}{m}$.  Also, note that:
  \begin{equation*}
    \tau_{\sigma,m+1}\circ \alpha_{\sigma,m}=\tau_{\sigma,m}
  \end{equation*}
  by \cite[Theorem V.3.6]{Davidson} and its proof.
\end{notation}

At the core of our construction of a quantum metric on the Bunce-Deddens algebras $\BunceDeddens{\sigma}$ lies a conditional expectation from a circle algebra to the image by a connecting morphism of a previous circle algebra in the inductive sequence defining $\BunceDeddens{\sigma}$. We now construct this conditional expectation.

\begin{notation}
  If $M \in \M_{\boxtimes\sigma_m}(\C)$, and if we write $M = \begin{pmatrix} N_{1,1} & \cdots & N_{1,\sigma_m} \\ \vdots & & \vdots \\ N_{\sigma_m,1} & \cdots & N_{\sigma_m,\sigma_m} \end{pmatrix}$ with $N_{1,1}$, \ldots, $N_{\sigma_m,\sigma_m}$ all in $\M_{\boxtimes\sigma_{m-1}}(\C)$, then we define:
  \begin{equation*}
    D_{\sigma,m}(M) = \begin{pmatrix} N_{1,1} & & & \\ & N_{2,2} & & \\ & & \ddots & \\ & & & N_{\sigma_m,\sigma_m}  \end{pmatrix} \text{.}
  \end{equation*}

  Modeled on \cite{Aguilar-Latremoliere15}, the map $D_{\sigma,m}$ is a conditional expectation from $\M_{\boxtimes\sigma(m)}(\C)$ to the C*-subalgebra of block-diagonal matrices with blocks all square matrices of order $\boxtimes\sigma_{m-1}$.
\end{notation}

\begin{lemma}\label{cond-exp-mid-lemma}
  Let $\sigma \in \BaireSpace$ and $m\in\N\setminus\{0\}$. If for all  $a\in\CP{\sigma}{m}$, we define:
\begin{equation*}
\CondExp{\sigma,m}{a} : t \in \R \mapsto U_{\sigma,m}(t) \left[ D_{\sigma,m} \left(U_{\sigma,m}(t)^\ast a(t)U_{\sigma,m}(t)\right)\right] U_{\sigma,m}(t)^\ast\text{,}
\end{equation*}

then  $\mathds{E}_{\sigma,m}: a\in \CP{\sigma}{m}\mapsto \CondExp{\sigma,m}{a}\in \alpha_{\sigma,m-1}(\CP{\sigma}{m-1})$ is a conditional expectation   onto $\alpha_{\sigma,m-1}(\CP{\sigma}{m-1})$   such that $ \tau_{\sigma,m}\circ \mathds{E}_{\sigma,m} =\tau_{\sigma,m}$.
\end{lemma}

\begin{proof}
  We use the notations of Lemma (\ref{lemma:block-unitary}) and use the computations of Lemma (\ref{lemma:alpha}).
  
  If $a = ( a_{j,k} )_{1\leq j,k \leq \sigma_m} \in \M_{\boxtimes\sigma_m}(\C)$, with $a_{1,1}$, \ldots, $a_{\sigma_m,\sigma_m}$ elements of $\M_{\boxtimes\sigma_{m-1}}(\C)$, then $W_{\sigma,m}^\ast a W_{\sigma,m}$ is the matrix $( a_{\pi(j),\pi(k)} )_{1\leq j,k \leq \sigma_m}$ where $\pi$ is the permutation $(2 \, 3 \, \ldots \, m \, 1)$. Thus in particular, \[D_{\sigma,m}(W_{\sigma,m}^\ast a W_{\sigma,m})=\begin{pmatrix} a_{2,2} & & & & \\ & a_{3,3} & & & \\ & & \ddots & & \\ & & & a_{\sigma_m,\sigma_m} & \\ & & & & a_{1,1} \end{pmatrix} = W_{\sigma,m}^\ast D_{\sigma,m}(a) W_{\sigma,m}.\] Hence  $D_{\sigma,m}$ commutes with $\mathrm{Ad}_{W_{\sigma,m}}$ (and similarly with $\mathrm{Ad}_{W_{\sigma,m}^\ast}$). Thus
  \begin{equation*}
    W_{\sigma,m} D_{\sigma,m} (W_{\sigma,m}^\ast a W_{\sigma,m}) W_{\sigma,m}^\ast = \begin{pmatrix} a_{1,1} & & & \\ & a_{2,2} \\ & & \ddots & \\ & & & a_{\sigma_m,\sigma_m} \end{pmatrix} = D_{\sigma,m}(a) \text{.}
  \end{equation*}

  By Lemma (\ref{lemma:block-unitary}), we also have:
  \begin{equation*}
    \forall t\in \R \quad U_{\sigma,m}(t+1) = U_{\sigma,m}(t) W_{\sigma,m} \text{.}
  \end{equation*}
  
  Let now $a\in \CP{\sigma}{m}$ --- note that $a$ is $1$-periodic. We then compute for any $t\in\R$:
  \begin{multline*}
    U_{\sigma,m}(t+1) \left[ D_{\sigma,m}(U_{\sigma,m}^\ast(t+1)\cdot a(t+1) \cdot U_{\sigma,m}(t+1)^\ast) \right] U_{\sigma,m}^\ast(t+1)\\
    \begin{aligned}
    & =   U_{\sigma,m}(t)W_{\sigma,m} \left[ D_{\sigma,m}(W_{\sigma,m}^\ast U_{\sigma,m}^\ast(t)\cdot a(t) \cdot U_{\sigma,m}(t)W_{\sigma,m}) \right] W_{\sigma,m}^\ast U_{\sigma,m}^\ast(t)\\
     & =   U_{\sigma,m}(t) \left[ D_{\sigma,m}(  U_{\sigma,m}^\ast(t)\cdot a(t) \cdot U_{\sigma,m}(t) ) \right]   U_{\sigma,m}^\ast(t)  \text{.}
    \end{aligned}
  \end{multline*}

  Therefore, $\mathds{E}_{\sigma,m}(a)$ is $1$-periodic, and obviously continuous, so it is an element of $\CP{\sigma}{m}$.
  
  Now, we wish to find $f \in \CP{\sigma}{m-1}$ such that $\alpha_{\sigma,m-1}(f)=\mathds{E}_{\sigma,m}(a)$. In particular, we wish to find $f \in \CP{\sigma}{m-1}$ such that:
  \begin{equation*}\label{eq:diag-reconstruction}
    \forall t \in \R \quad \begin{pmatrix} f\left(\frac{t}{\sigma_{m}}\right) & & & \\
      & f\left(\frac{t+1}{\sigma_{m}}\right) & & \\ & & \ddots & \\ & & & f\left(\frac{t+\sigma_{m}-1}{\sigma_{m}}\right) \end{pmatrix}  = D_{\sigma,m}( U_{\sigma,m}^\ast(t) \cdot a(t) \cdot U_{\sigma,m}(t))\text{.}
  \end{equation*}
  
  To this end,  if $M \in \M_{\boxtimes\sigma_m}(\C)$, and if we write $M = \begin{pmatrix} N_{1,1} & \cdots & N_{1,\sigma_m} \\ \vdots & & \vdots \\ N_{\sigma_m,1} & \cdots & N_{\sigma_m,\sigma_m} \end{pmatrix}$ with $N_{1,1}$, \ldots, $N_{\sigma_m,\sigma_m}$ all in $\M_{\boxtimes\sigma_{m-1}}(\C)$, then for $j \in \{1, \ldots,\sigma_{m}\}$,  we define:
  \begin{equation*}
    F_{\sigma,m,j}(M) =N_{j,j}, %\begin{pmatrix} 0 & & & & & &  \\ & \ddots & & & &  & \\ & & 0 & & &  & \\ & & & N_{j,j} & & &  \\ & & & &0  & & \\ & & & & &\ddots & \\ & & & & & & 0  \end{pmatrix} \text{,}
  \end{equation*}
   Modeled on \cite{Aguilar-Latremoliere15}, the map $F_{\sigma,m,j}$ is a unital completely positive contraction from  $\M_{\boxtimes\sigma(m)}(\C)$ onto $\M_{\boxtimes\sigma(m-1)}(\C)$. Also, a similar calculation to the one involving $D_{\sigma,m}$ shows that $F_{\sigma,m,j}(W_{\sigma,m}^\ast MW_{\sigma,m})=F_{\sigma,m,j+1}(M)$ for $j \in \{1, \ldots, \sigma_m-1\}$ and $F_{\sigma,m,\sigma_m}(W_{\sigma,m}^\ast MW_{\sigma,m})=F_{\sigma,m,1}(M)$.

  Next, fix $j \in \{0, \ldots, \sigma_m-1\}$. For all $t \in \left[\frac{j}{\sigma_m}, \frac{j+1}{\sigma_m}\right]$, set 
  \[f_j(t)=F_{\sigma,m,j+1}(D_{\sigma,m}( U_{\sigma,m}^\ast(t\sigma_m-j) \cdot a(t\sigma_m-j) \cdot U_{\sigma,m}(t\sigma_m-j))).\]
  Since $F_{\sigma,m,j+1}$ is continuous, then so is $f_j$ on $\left[\frac{j}{\sigma_m}, \frac{j+1}{\sigma_m}\right]$.
  
  Now, let $j \in \{0, \ldots, \sigma_{m}-1\}$. We have since $a$ is $1$-periodic
  \begin{equation*}
  \begin{split}
  f_{j}\left(\frac{j+1}{\sigma_m} \right) 
  & = F_{\sigma,m,j+1}(D_{\sigma,m}( U_{\sigma,m}^\ast(1) \cdot a(1) \cdot U_{\sigma,m}(1)))\\
  & = F_{\sigma,m,j+1}(D_{\sigma,m}( W_{\sigma,m}^\ast U_{\sigma,m}^\ast(0) \cdot a(0) \cdot U_{\sigma,m}(0)W_{\sigma,m}))\\
  &=F_{\sigma,m,j+1}(W_{\sigma,m}^\ast D_{\sigma,m}(  U_{\sigma,m}^\ast(0) \cdot a(0) \cdot U_{\sigma,m}(0))W_{\sigma,m})\\
  &=F_{\sigma,m,j+2}(  D_{\sigma,m}(  U_{\sigma,m}^\ast(0) \cdot a(0) \cdot U_{\sigma,m}(0)) )  = f_{j+1}\left(\frac{j+1}{\sigma_m} \right).
  \end{split}
  \end{equation*}
   A similar computation shows that $f_{\sigma_m}(1) = f_0(0)$.  Thus, the map defined for all $t \in [0,1]$ by
   \[
   f(t)=\begin{cases}
   f_j(t) & \text{ if } t \in \left[\frac{j}{\sigma_m}, \frac{j+1}{\sigma_m}\right] \ \land \ j \in \{0, \ldots, \sigma_m-1\},
   \end{cases}
   \]
   is well-defined and continuous and $f(0)=f(1)$. Thus, $f$ extends uniquely to an element in $\CP{\sigma}{m-1}$, which we will still denote by $f$. And, by construction, we have that  $\alpha_{\sigma,m-1}(f)=\mathds{E}_{\sigma,m}(a)$.
   
   Next, it remains to show that if $b \in \alpha_{\sigma,m-1}(\CP{\sigma}{m-1})$, then $\mathds{E}_{\sigma,m}(b)=b$. Let $b \in \alpha_{\sigma,m-1}(\CP{\sigma}{m-1})$. Thus, there exists $c \in \CP{\sigma}{m-1}$ such that $b=\alpha_{\sigma,m-1}(c)$.
   
   Hence,  for all $t \in \R$, we have
   \begin{equation*}
   \begin{split}
   & \mathds{E}_{\sigma,m}(b)(t)\\
   &=U_{\sigma,m}(t) \left[ D_{\sigma,m} \left(U_{\sigma,m}(t)^\ast \alpha_{\sigma,m-1}(c)(t)U_{\sigma,m}(t)\right)\right] U_{\sigma,m}(t)^\ast\\
   &= U_{\sigma,m}(t) \left[ D_{\sigma,m}   \begin{pmatrix} c\left(\frac{t}{\sigma_{m}}\right) & & & \\ & c\left(\frac{t+1}{\sigma_{m}}\right) & & \\ & & \ddots & \\ & & & c\left(\frac{t + \sigma_{m}-1}{\sigma_{m}}\right) \end{pmatrix}  \right] U_{\sigma,m}(t)^\ast\\
   & =U_{\sigma,m}(t)   \begin{pmatrix} c\left(\frac{t}{\sigma_{m}}\right) & & & \\ & c\left(\frac{t+1}{\sigma_{m}}\right) & & \\ & & \ddots & \\ & & & c\left(\frac{t + \sigma_{m}-1}{\sigma_{m}}\right) \end{pmatrix}   U_{\sigma,m}(t)^\ast\\
   & = \alpha_{\sigma,m-1}(c)(t)=b(t).
   \end{split}
   \end{equation*}

 Now, $\mathds{E}_{\sigma,m}$ is positive and contractive by construction, as the composition of *-isomorphisms and a conditional expectation. So  $\mathds{E}_{\sigma,m}$ is a conditional expectation onto $\alpha_{\sigma,m-1}(\CP{\sigma}{m-1})$ by \cite[Tomiyama's Theorem 1.5.10 and Theorem 3.5.3]{Brown-Ozawa}.

  Finally, following \cite[Theorem V.3.6]{Davidson}, we have since $D_{\sigma,m}$ preserves the  trace, $\mathrm{Tr}$:
  \begin{equation*}
  \begin{split}
  \tau_{\sigma,m}(\mathds{E}_{\sigma,m}(a)) &= \int_0^1 \boxtimes \sigma_m^{-1} \mathrm{Tr} (\mathds{E}_{\sigma,m}(a)(t)) \ dt \\
  &=  \int_0^1 \boxtimes \sigma_m^{-1} \mathrm{Tr} (U_{\sigma,m}(t) \left[ D_{\sigma,m} \left(U_{\sigma,m}(t)^\ast a(t)U_{\sigma,m}(t)\right)\right] U_{\sigma,m}(t)^\ast) \ dt \\
&=    \int_0^1 \boxtimes \sigma_m^{-1} \mathrm{Tr}   (D_{\sigma,m} \left(U_{\sigma,m}(t)^\ast a(t)U_{\sigma,m}(t)\right) ) \ dt \\
& = \int_0^1 \boxtimes \sigma_m^{-1} \mathrm{Tr}    \left(U_{\sigma,m}(t)^\ast a(t)U_{\sigma,m}(t)\right)  \ dt \\
&=  \int_0^1 \boxtimes \sigma_m^{-1} \mathrm{Tr}    \left(  a(t) \right)  \ dt   = \tau_{\sigma,m}( a),
\end{split}
  \end{equation*} 
  which completes the proof. 
\end{proof}

\section{The metric geometry of the class of the Bunce-Deddens Algebras}

In this section, we construct our Lip-norms on   circle algebras that are meant to be suitable with both the inductive limit structure and the conditional expectations presented in the previous section. This will then allow us to utilize Theorem \ref{theorem:inductive} to get one step closer to building Lip-norms on the Bunce-Deddens algebras. We begin with some classical structure.

\begin{notation}\label{n-Lip-def}
  Let $(X, \mathsf{d})$ be a locally compact metric space, and let $n \in \N \setminus \{0\}$.  Let $\B$ be a unital C*-subalgebra of the C*-algebra $C_b(X, \M_n(\C))$ of bounded $\M_n(\C)$-valued continuous functions over $X$, such that the unit $1_\B$ is the unit of $C_b(X,\M_n(\C))$ --- the constant function equal to the identity in $\M_n(\C)$.

  For all $a \in \B$, we define:
  \begin{equation*}
    l_{\mathsf{d}}^n (a)= \sup_{x,y\in X, x \neq y} \frac{\|a(x)-a(y)\|_{\M_n(\C)}}{\mathsf{d}(x,y)}\text{.}
  \end{equation*}  
  
  Last, if $X$ is a normed vector space with norm $\norm{\cdot}{X}$, then we write $l_{\norm{\cdot}{X}}^{n}$ for $l_{d}^{n}$ with $d$ the metric induced by $\norm{\cdot}{X}$.
  
  We make two simple but important remarks:
\begin{itemize}
\item $\forall a \in \B \quad l_d^n(a^\ast) = l_d^n(a)$,
\item $\forall a,b \in \B \quad l_d^n(ab) \leq \norm{a}{\B} l_d^n(b) + l_d^n(a) \norm{b}{\B}$.
\end{itemize}
\end{notation}

\begin{lemma}\label{lemma:unitary-lip}
  Let $m \in \N \setminus \{0\}$. We use Notation (\ref{notation:z}).
  \begin{itemize}
  \item For all $j,k \in \{1, \ldots, m\}$, we estimate:
    \begin{equation*}
      l_{|\cdot|} (z^m_{j,k}) \leq \frac{m-j}{m^{3/2}}\text{.}
    \end{equation*}
    
  \item We have:
    \begin{equation*}
       l_{|\cdot|}^m (U_m) \leq \sqrt{\frac{2m^2+3m+1}{6 m}}.
    \end{equation*}
  \end{itemize}
\end{lemma}

\begin{proof}
Fix $m\in\N\setminus\{0\}$ and $j,k \in \{1,\ldots,n\}$. Of course $l_{|\cdot|}(z) \leq 1$.  Hence, if $r,t \in \R$, then:
\begin{align*}
& \left| z^m_{j,k}(t) - z^m_{j,k}(r) \right| \\
&= \left\vert \frac{1}{\sqrt{m}} z\left(\frac{(m-j)(t+k-1)}{m}\right) - \frac{1}{\sqrt{m}} z\left(\frac{(m-j)(r+k-1)}{m}\right) \right\vert \\
&\leq \frac{1}{\sqrt{m}} \left\vert \frac{(m-j)(t+k-1)}{m}-\frac{(m-j)(r+k-1)}{m}\right\vert \\
& \leq \frac{1}{\sqrt{m}} \left( \frac{m-j}{m}\right) \vert t-r\vert. 
\end{align*}
Thus $l_{|\cdot|} (z^m_{j,k}) \leq \frac{m-j}{m^{3/2}}$.

  Let $t,s\in\R$. Let $\xi = (\xi_1,\ldots,\xi_m)$ with $\norm{\xi}{2} = \sqrt{\sum_{j=1}^m |\xi_j|^2} \leq 1$. We compute:
  \begin{equation*}
    (U_m(t)-U_m(s))\xi = \left( \sum_{k=1}^m (z_{j,k}(t)-z_{j,k}(s)) \xi_k \right)_{j \in \{1,\ldots,m\}} \text{.}
  \end{equation*}

  Now for all $j\in \{1,\ldots,m\}$:
  \begin{align*}
    \left|\sum_{k=1}^m (z_{j,k}(t)-z_{j,k}(s)) \xi_k \right|^2
    &\leq \left(\sum_{k=1}^m \left|z_{j,k}(t)-z_{j,k}(s)\right|^2\right) \left(\sum_{k=1}^m |\xi_k|^2\right)\\
    &\leq \left(\sum_{k=1}^m \left( \frac{m-j}{m\sqrt{m}} |t-s|\right)^2\right)\cdot 1\\
    &\leq \frac{(m-j)^2}{m^2}|t-s|^2 
    \text{.}
  \end{align*}

  Now
  \begin{align*}
    \norm{(U_m(t)-U_m(s))\xi}{2}
    &\leq  \sqrt{\sum_{j=1}^m \frac{(m-j)^2}{m^2}|t-s|^2} \\
    &\leq  |t-s| \sqrt{\frac{1}{m^2} + \frac{4}{m^2} + \ldots + 1} \\
    &=  |t-s| \sqrt{\frac{(2m+1)(m+1)m}{6 m^2}}\\
    &= |t-s| \sqrt{\frac{2m^2+3m+1}{6 m}}\text{.}
  \end{align*}

  Thus $\sup_{\norm{\xi}{2}\leq 1} \norm{(U_m(t) - U_m(s))\xi}{2}\leq \sqrt{\frac{2m^2+3m+1}{6 m}}|t-s|$. Hence \[\norm{U_m(s)-U_m(t)}{\M_m(\C)} \leq \sqrt{\frac{2m^2+3m+1}{6 m}}|t-s|.\] This concludes our proof.
\end{proof}

We first define a natural $(2,0)$-Leibniz Lip-norm on the circle algebras.

\begin{definition}\label{def:lip}
  Let $\sigma\in\BaireSpace$ and $m\in\N\setminus\{0\}$. We define $ \forall a \in \sa{\CP{\sigma}{m}}$:
  \begin{equation*}
     \Lip_{\sigma,m}(a) = \max\left\{ l_{|\cdot|}^{\boxtimes\sigma_m}(U_{\sigma,m}^\ast a U_{\sigma,m}), \norm{a - \tau_{m,\sigma}(a)1_{\CP{\sigma}{m}}}{\CP{\sigma}{m}}\right\}\text{.}
  \end{equation*}
\end{definition}

The main motivation for our choice of Lip-norm is that it is well-adapted to the conditional expectation. Before we prove that Definition (\ref{def:lip}) actually gives Lip-norms on circle algebras, we prove the following key result.

\begin{lemma}\label{lemma:key}
  Let $\sigma\in\BaireSpace$ and let $m\in\N\setminus\{0\}$. If $a\in \CP{\sigma}{m}$ then:
  \begin{equation*}
    \Lip_{\sigma,m}(\CondExp{\sigma,m}{a}) \leq \Lip_{\sigma,m}(a) \text{.}
  \end{equation*}
\end{lemma}

\begin{proof}
  Let $n = \boxtimes\sigma_m$ and $a\in \CP{\sigma}{m}$. To ease notations, we just write $l$ for the seminorm $l_{|\cdot|}^{\boxtimes\sigma_m}$.
  
  Since $\opnorm{D_{\sigma,m}}{}{\M_{n}(\C)}\leq 1$, we compute:
  \begin{multline*}
    l\left( U_{\sigma,m}^\ast \CondExp{\sigma,m}{a} U_{\sigma,m} \right)\\
    \begin{aligned}
      &= l\left( D_{\sigma,m} \left( U_{\sigma,m}^\ast a U_{\sigma,m} \right) \right)\\
      &= \sup_{\substack{x,y\in \R \\ x\not=y}} \frac{\norm{D_{\sigma,m}\left(U_{\sigma,m}(x)^\ast a U_{\sigma,m}(x)\right) - D_{\sigma,m}\left(U_{\sigma,m}^\ast(y) a U_{\sigma,m}(y)\right)}{\M_{n}(\C)}}{|x-y|} \\
      &= \sup_{\substack{x,y\in \R \\ x\not=y}} \frac{\norm{D_{\sigma,m}\left(U_{\sigma,m}^\ast(x) a U_{\sigma,m}(x) - U_{\sigma,m}^\ast(y) a U_{\sigma,m}(y)\right)}{\M_{n}(\C)}}{|x-y|}  \\
      &\leq \sup_{\substack{x,y\in \R \\ x\not=y}}  \frac{\norm{U_{\sigma,m}^\ast(x) a U_{\sigma,m}(x) - U_{\sigma,m}^\ast(y) a U_{\sigma,m}(y)}{\M_{n}(\C)}}{|x-y|} \\
      &= l(U_{\sigma,m}^\ast a U_{\sigma,m}) \text{.}
    \end{aligned}
  \end{multline*}
  
  Furthermore, by Lemma \ref{cond-exp-mid-lemma}:
  \begin{align*}
   & \norm{ \CondExp{\sigma,m}{a} - \tau_{\sigma,m}(\CondExp{\sigma,m}{a}) 1_{\CP{\sigma}{m}}}{\CP{\sigma}{m}} \\
   &= \norm{ \CondExp{\sigma,m}{a - \tau_{\sigma,m}(a)1_{\CP{\sigma}{m}}}}{\CP{\sigma}{m}} \\ &\leq \norm{ a - \tau_{\sigma,m}(a)1_{\CP{\sigma}{m}} }{\CP{\sigma}{m}} \text{.}
  \end{align*}

  Therefore:
  \begin{align*}
   & \Lip_{\sigma,m}(E_{\sigma,m}(a))\\
      &= \max\big\{ l(U_{\sigma,m}^\ast \CondExp{\sigma,m}{a} U_{\sigma,m}), \norm{ \CondExp{\sigma,m}{a} - \tau_{\sigma,m}(\CondExp{\sigma,m}{a}) 1_{\CP{\sigma}{m}}}{\CP{\sigma}{m}} \big\} \\
      &\leq \max\big\{ l(U_{\sigma,m}^\ast a U_{\sigma,m}), \norm{a-\tau_{\sigma,m}(a) 1_{\CP{\sigma}{m}}}{\CP{\sigma}{m}} \big\} \\
      &= \Lip_{\sigma,m}(a) \text{.}
  \end{align*}
  This concludes our result.
\end{proof}

\begin{theorem}
  If $\sigma\in\BaireSpace$ and $m\in\N\setminus\{0\}$, then $\left(\CP{\sigma}{m},\Lip_{\sigma,m}\right)$ is a $(2,0)$-{\qcms}.
\end{theorem}

\begin{proof}

  As the maximum of two lower semi-continuous seminorms, $\Lip_{\sigma,m}$ is a lower semi-continuous seminorm (allowing for the value $\infty$).
  
  By \cite[Lemma 2.3]{Aguilar-Latremoliere15}, the seminorm $a\in \CP{\sigma}{m} \mapsto \norm{ a - \tau_{\sigma,m}(a)1_\CP{\sigma}{m} }{\CP{\sigma}{m}}$ is (2,0)-quasi-Leibniz. As $a\in \CP{\sigma}{m}\mapsto l_{|\cdot|}^{\boxtimes\sigma_m}(U_{\sigma,m}^\ast a U_{\sigma,m})$ is of course Leibniz, we conclude that $\Lip_{\sigma,m}$ is $(2,0)$-quasi-Leibniz.

  If $\Lip_{\sigma,m}(a) = 0$ then $l_{|\cdot|}^{\boxtimes\sigma_m}(U_{\sigma,m}^\ast a U_{\sigma,m}) = 0$, so there exists $T \in \M_{\boxtimes\sigma_m}(\C)$ such that $a = U_{\sigma,m} T U_{\sigma,m}^\ast$. On the other hand:
  \begin{align*}
    0
    &= \norm{ a - \tau_{\sigma,m}(a)1_{\CP{\sigma}{m}} }{\CP{\sigma}{m}} \\
    &= \norm{ U_{\sigma,m}^\ast (a - \tau_{\sigma,m}(a)1_{\CP{\sigma}{m}}) U_{\sigma,m}}{\CP{\sigma}{m}} \text{ since $U_{\sigma,m}$ is unitary,}\\
    &= \norm{ U_{\sigma,m}^\ast a U_{\sigma,n} - \tau_{\sigma,m}(a)1_{\CP{\sigma}{m}} }{\CP{\sigma}{m}}
  \end{align*}
  so $T = \tau_{\sigma,m}(a)\mathrm{id}_{\boxtimes \sigma_m}$, and therefore $a = \tau_{\sigma,m}(a)1_{\CP{\sigma}{m}} \in \C\unit_{\CP{\sigma}{m}}$. Of course, $\Lip_{\sigma,m}(\unit_{\CP{\sigma}{m}}) = 0$.

  If $f \in \CP{\sigma}{0}$ with $l_{|\cdot|}^1(f) < \infty$ and if $T \in\CP{\sigma}{m}$ then:
  \begin{equation}\label{eq:tensor-leibniz}
    l_{|\cdot|}^{\boxtimes\sigma_m}(f\otimes T)
    \leq l_{|\cdot|}^1(f) \norm{T}{\CP{\sigma}{m}} + \norm{f}{C_b(\R)} l_{|\cdot|}^{\boxtimes\sigma_m}(T)  
  \end{equation}
  using the standard *-isomorphism between $\CP{\sigma}{0}\otimes \M_{\boxtimes\sigma_m}(\C)$ and $\CP{\sigma}{m}$ given on elementary tensors by $f \otimes T  \in \CP{\sigma}{0}\otimes \M_{\boxtimes\sigma_m}(\C) \mapsto (t \in \R \mapsto f(t)T \in \M_{\boxtimes\sigma_m}(\C)  )\in \CP{\sigma}{m}$ \cite[Theorem 6.4.17]{Murphy90}.
  
  Using the same *-isomorphism, if $a \in \CP{\sigma}{m}$ and $\varepsilon > 0$, then there exist $f_1,\ldots,f_k \in \CP{\sigma}{0}$ and $T_1,\ldots,T_k \in \M_{\boxtimes\sigma_m}(\C)$ such that:
  \begin{equation*}
    \norm{ a - \sum_{j=1}^k f_j \otimes T_j }{\CP{\sigma}{m}}{} < \frac{\varepsilon}{2}.  
  \end{equation*}

  Let $K = k \max \left\{ \norm{T_j}{M_{\boxtimes\sigma_m}(\C)} : j\in\{1,\ldots,k\} \right\}$.
  
  As Lipschitz functions are dense in $\CP{\sigma}{0}$, there exists $g_1,\ldots,g_k \in \CP{\sigma}{0}$ such that $\norm{f_j - g_j}{\CP{\sigma}{0}} < \frac{\varepsilon}{2 K}$ while $l_{|\cdot|}^1(g_j) < \infty$ for all $j\in\{1,\ldots,k\}$. Therefore:
  \begin{align*}
    \norm{ a - \sum_{j=1}^k g_j\otimes T_j }{\CP{\sigma}{m}}
      &\leq \norm{ a - \sum_{j=1}^k f_j\otimes T_j }{\CP{\sigma}{m}}\\
      & \quad \quad  + \norm{\sum_{j=1}^k f_j\otimes T_j  - \sum_{j=1}^k g_j\otimes T_j }{\CP{\sigma}{m}}\\
      &\leq \frac{\varepsilon}{2} + \sum_{j=1}^k \norm{f_j-g_j}{\CP{\sigma}{0}} K \leq \varepsilon \text{.}
    \end{align*}

  Last, using the Leibniz property of $l_{|\cdot|}^{\boxtimes\sigma_m}$ and the fact that $l_{|\cdot|}^{\boxtimes\sigma_m} (U_{\sigma,m}) < \infty$, we conclude that $l_{|\cdot|}^{\boxtimes\sigma_m}(U_{\sigma,m}^*\sum_{j=1}^k (g_j\otimes T_j) U_{\sigma,m}) < \infty$ by Expression \eqref{eq:tensor-leibniz}. This concludes the proof that the domain of $\Lip_{\sigma,m}$ is dense in $\CP{\sigma}{m}$ since the other seminorm in its definition are actually continuous on $\CP{\sigma}{m}$.

  Last, let $(a_n)_{n\in\N}$ be a sequence in $ \CP{\sigma}{m}$ such that $\Lip_{\sigma,m}(a_n) \leq 1$ and $\tau_{\sigma,m}(a_n) = 0$. for all $n\in\N$. Since $l_{|\cdot|}^{\boxtimes\sigma_m}(a_n) \leq 1$ for all $n\in\N$, the set $\{ U_{\sigma,m}^\ast a_n U_{\sigma,m} : n \in \N \}$ is equicontinuous, and thus $\{ a_n : n \in \N \}$ is equicontinuous as $U_{\sigma,m}$ is unitary. Moreover, we also have:
  \begin{equation*}
    \forall n \in \N \quad \norm{a_n}{\CP{\sigma}{m}} \leq \norm{a_n-\tau_{\sigma,m}(a_n)1_{\CP{\sigma}{m}}}{\CP{\sigma}{m}} \leq 1\text{.} 
  \end{equation*}
  So $\{ a_n : n \in \N \}$ is an equicontinuous set of continuous $1$-periodic functions on $\R$, all valued in the closed unit disk, which is compact. By Arz{\'e}la-Ascoli theorem, we thus conclude that $\{ a_n : n \in \N \}$ is totally bounded for the norm (note: apply Arz{\'e}la-Ascoli theorem to the restriction of these functions to the compact $[0,1]$ and then conclude using periodicity). Therefore, the sequence $(a_n)_{n\in\N}$ admits a Cauchy subsequence $(a_{r(n)})_{n\in\N}$. As $\CP{\sigma}{m}$ is complete, $(a_{r(n)})_{n\in\N}$ converges to some $a\in \CP{\sigma}{m}$. As $\Lip_{\sigma,m}$ is lower semi-continuous, we get $\Lip_{\sigma,m}(a) \leq 1$.

  Thus $\{ a \in \CP{\sigma}{m} : \Lip_{\sigma,m}(a) \leq 1, \tau_{\sigma,m}(a) = 0 \}$ is compact. By \cite[Proposition 1.3]{Ozawa05}, we thus can conclude that $\left(\CP{\sigma}{m},\Lip_{\sigma,m}\right)$ is a {\qcms}.
\end{proof}

Now, we study the metric properties of the connecting maps defining the Bunce-Deddens algebras.

\begin{lemma}
  Let $\sigma\in\BaireSpace$ and $m\in\N\setminus\{0\}$. Let $k_m =\max\left\{1, \frac{1+2 l_{|\cdot|}^{\boxtimes\sigma_{m-1}}(U_{\sigma,m-1})}{\sigma_m}\right\}$. If $a\in \CP{\sigma}{m-1}$ then:
  \begin{equation*}
    \frac{1}{\sigma_m^2 k_m} \Lip_{\sigma,m-1}(a) \leq \Lip_{\sigma,m}(\alpha_{\sigma,m-1}(a)) \leq k_m  \Lip_{\sigma,m-1}(a) \text{.}
  \end{equation*}
\end{lemma}

\begin{proof}

  We denote $l_{|\cdot|}^{\boxtimes\sigma_m}$ by $l$ and $l_{|\cdot|}^{\boxtimes\sigma_{m-1}}$ by $l_{-1}$ in this proof.
  
 For $a \in \CP{\sigma}{m-1}$,  set:
  \begin{equation*}
    \theta(a) : t \in \R \mapsto
    \begin{pmatrix}
      a\left(\frac{t}{\sigma_m}\right) & & &  \\
      & a\left(\frac{t+1}{\sigma_m}\right) & & \\
      & & \ddots & \\
      & & & a\left(\frac{t + \sigma_m - 1}{\sigma_m}\right) 
    \end{pmatrix}\in \M_{\boxtimes\sigma_m}(\C)\text{.}
  \end{equation*}

  We compute:
  \begin{align*}
    l(\theta(a)) &= \sup_{\substack{ x,y \in \R \\ x\not= y}}\left\{ \frac{\norm{\theta(a)(x)-\theta(a)(y)}{\M_{\boxtimes\sigma_m}(\C)}}{|x-y|} \right\} \\
    &= \sup_{\substack{ x,y \in \R \\ x\not= y}}\left\{ \max_{j\in\{0,\ldots,\sigma_m-1\}} \frac{\norm{a\left(\frac{x + j}{\sigma_m}\right)-a\left(\frac{y+j}{\sigma_m}\right)}{\M_{\boxtimes\sigma_{m-1}}(\C)}}{|x-y|} \right\}\\
    &=\frac{1}{\sigma_m}\cdot\sup_{\substack{ x,y \in \R \\ x\not= y}}\left\{ \max_{j\in\{0,\ldots,\sigma_m-1\}} \frac{\norm{a\left(\frac{x + j}{\sigma_m}\right)-a\left(\frac{y+j}{\sigma_m}\right)}{\M_{\boxtimes\sigma_{m-1}}(\C)}}{\frac{1}{\sigma_m}|x-y|} \right\}\\
    &=\frac{1}{\sigma_m}\cdot\sup_{\substack{ x,y \in \R \\ x\not= y}}\left\{ \max_{j\in\{0,\ldots,\sigma_m-1\}} \frac{\norm{a\left(\frac{x + j}{\sigma_m}\right)-a\left(\frac{y+j}{\sigma_m}\right)}{\M_{\boxtimes\sigma_{m-1}}(\C)}}{ \left|\frac{x+j}{\sigma_m}-\frac{y+j}{\sigma_m}\right|} \right\}\\
    &= \frac{1}{\sigma_m} l_{-1}(a) \text{.}
  \end{align*}
  
  Therefore:
  \begin{align*}
    &l (U_{\sigma,m}^\ast \alpha_{\sigma,m-1}(a) U_{n,m}) \\
      &= l (\theta(a)) = \frac{1}{\sigma_m} l_{-1} (a)= \frac{1}{\sigma_m} l_{-1} (U_{\sigma,m-1}U_{\sigma,m-1}^\ast a U_{\sigma,m-1}U_{\sigma,m-1}^\ast) \\
      &\leq \frac{1}{\sigma_m}\left( l_{-1}(U_{\sigma,m-1}^\ast a U_{\sigma,m-1}) + 2 l_{-1}(U_{\sigma,m-1}) \norm{U_{\sigma,m-1}^\ast a U_{\sigma,m-1}}{C_b(\R, \M_{\boxtimes\sigma_{m-1}}(\C))} \right)\\
      & \leq \frac{1}{\sigma_m}\left( \Lip_{\sigma,m-1}(a) + 2 l_{-1}(U_{\sigma,m-1}) \norm{U_{\sigma,m-1}^\ast a U_{\sigma,m-1}}{C_b(\R, \M_{\boxtimes\sigma_{m-1}}(\C))} \right)\text{.}
    \end{align*}

  Now, 
  \begin{equation*}
  \begin{split}
  &\norm{U_{\sigma,m-1}^\ast (a-\tau_{\sigma,m-1}( a)1_{\CP{\sigma}{m-1}}) U_{\sigma,m-1}}{C_b(\R, \M_{\boxtimes\sigma_{m-1}}(\C))}\\
  & \leq \norm{ a-\tau_{\sigma,m-1}( a)1_{\CP{\sigma}{m-1}} }{\CP{\sigma}{m-1}}\\
  & \leq \Lip_{\sigma,m-1}(a).
  \end{split}
  \end{equation*}
  Thus since the seminorms $l$ and $\Lip_{\sigma,m-1}$ vanish on scalars, we have
  \begin{equation*}
  \begin{split}
  & l (U_{\sigma,m}^\ast \alpha_{\sigma,m-1}(a) U_{n,m})\\
  & = l (U_{\sigma,m}^\ast \alpha_{\sigma,m-1}(a-\tau_{\sigma,m-1}( a)1_{\CP{\sigma}{m-1}}) U_{n,m})\\
  &  \leq \frac{1}{\sigma_m}\Big( \Lip_{\sigma,m-1}(a-\tau_{\sigma,m-1}( a)1_{\CP{\sigma}{m-1}}) \\
  & \quad \quad + 2 l_{-1}(U_{\sigma,m-1}) \norm{U_{\sigma,m-1}^\ast (a-\tau_{\sigma,m-1}( a)1_{\CP{\sigma}{m-1}}) U_{\sigma,m-1}}{C_b(\R, \M_{\boxtimes\sigma_{m-1}}(\C))} \Big)\\
  & \leq \frac{1}{\sigma_m}\left( \Lip_{\sigma,m-1}(a ) + 2 l_{-1}(U_{\sigma,m-1}) \Lip_{\sigma,m-1}(a) \right)\\
  & =\frac{1+2 l_{-1}(U_{\sigma,m-1})}{\sigma_m}\Lip_{\sigma,m-1}(a).
  \end{split}
  \end{equation*}
  Next, we note that
  \begin{equation*}
  \begin{split}
  &\norm{\alpha_{\sigma,m-1}(a)-\tau_{\sigma,m}(\alpha_{\sigma,m-1}(a))1_{\CP{\sigma}{m}}}{\CP{\sigma}{m}}\\
&=\norm{ a-\tau_{\sigma,m-1}( a)1_{\CP{\sigma}{m-1}}}{\CP{\sigma}{m-1}} \leq \Lip_{\sigma,m-1}(a).
  \end{split}
  \end{equation*}

  Since $\norm{\alpha(c) - \tau_{\sigma,n}(\alpha_{\sigma,m}(c))}{\CP{\sigma}{n}} = \norm{c - \tau_{\sigma,m-1}(c)}{\CP{\sigma}{m-1}}$ by construction, we conclude:
  \begin{equation*}
    \Lip_{\sigma,m}(\alpha_{\sigma,m}(a)) \leq k_m  \Lip_{\sigma,m-1}(a) \text{.}
  \end{equation*}

  Let $c\in\CP{\sigma}{m-1}$ such that $\Lip_{\sigma,m}(\alpha_{\sigma,m-1}(c)) = 1$. By the above computation, we see that $l_{-1}(c) \leq \sigma_m$ and $\norm{c - \tau_{\sigma,m-1}(c)1_{\CP{\sigma}{m-1}}}{\CP{\sigma}{m-1}}\leq 1$. So:
  \begin{equation*}
  \begin{split}
    l_{-1}(U_{\sigma,m-1}^\ast c U_{\sigma,m-1}) 
       \leq 2l_{-1}(U_{\sigma,m-1})\|c\|_{\CP{\sigma}{m-1}} +l_{-1}(c).
    \end{split}
    \end{equation*}
 Thus, since $l_{-1}$ vanishes on scalars, we have
 \begin{equation*}
 \begin{split}
 &l_{-1}(U_{\sigma,m-1}^\ast c U_{\sigma,m-1})\\
 & = l_{-1}(U_{\sigma,m-1}^\ast ( c - \tau_{\sigma,m-1}(c)1_{\CP{\sigma}{m-1}} ) U_{\sigma,m-1})\\
 & \leq   2l_{-1}(U_{\sigma,m-1})\|c - \tau_{\sigma,m-1}(c)1_{\CP{\sigma}{m-1}}\|_{\CP{\sigma}{m-1}} +l_{-1}(c)\\
 & \leq 2l_{-1}(U_{\sigma,m-1}) +l_{-1}(c)\\
 & \leq 2l_{-1}(U_{\sigma,m-1})\sigma_m +\sigma_m= (1+ 2l_{-1}(U_{\sigma,m-1}))\sigma_m.
 \end{split}
 \end{equation*}

  As above, we conclude:
  \begin{equation*}
    \forall c \in  \CP{\sigma}{m-1} \quad \Lip_{\sigma,m-1}(c) \leq (1+2 l_{-1}(U_{\sigma,m-1})) \sigma_m \Lip_{\sigma,m}(\alpha_{\sigma,m-1}(c)) \text{.}
  \end{equation*}
  This concludes our proof.
\end{proof}

\begin{theorem}\label{theorem:cauchy}
  If $\sigma \in \BaireSpace$, and if for all $m\in\N$ we set $\mathsf{S}_{\sigma,0} = \Lip_{\sigma,0}$ on $\sa{\CP{\sigma}{0}}$ and for all $m\in\N\setminus\{0\}$:
  \begin{multline*} 
    \forall a \in \sa{\CP{\sigma}{m}} \quad \mathsf{S}_{\sigma,m}(a) = \\
    \max\left\{ \varkappa_m \Lip_{\sigma,m}(a) , \mathsf{S}_{\sigma,m-1}\circ\alpha_{\sigma,m-1}^{-1}\circ\CondExp{\sigma,m}{a} , \frac{1}{2^m}\norm{a - \CondExp{\sigma,m}{a}}{\CP{\sigma}{m}}  \right\}
  \end{multline*}
  where for all $m\in \N\setminus\{0\}$, we have $k_m= \frac{1 + 2 l_{|\cdot|}^{\boxtimes\sigma_{m-1}}(U_{\sigma,m-1})}{\sigma_m}$, and:
  \begin{equation*}
    \forall n\in \N \quad \varkappa_n = \begin{cases}
      1 \text{ if $n \in \{0, 1\}$,}\\
      \frac{\varkappa_{n-1}}{k_n} \text{ otherwise}
    \end{cases}
  \end{equation*}
  then $(\sa{\CP{\sigma}{m}},\mathsf{S}_{\sigma,m})$ is a {\ouqcms} and there exists a quantum metric order unit space $(O(\sigma),\mathsf{S})$ such that:
  \begin{equation*}
    \lim_{m\rightarrow\infty} \dist_q(O(\sigma),\mathsf{S}), (\sa{\CP{\sigma}{m}},\mathsf{S}_{\sigma,m})) = 0 \text{.}
  \end{equation*}
\end{theorem}

\begin{proof}
  We apply Theorem (\ref{theorem:inductive}) to the sequence $((\CP{\sigma}{m},\Lip_{\sigma,m}),\alpha_{\sigma,m})_{m\in\N}$ and the conditional expectations $\mathds{E}_{\sigma,m}$.
\end{proof}

Of course, we now want to show that $O(\sigma)$, as defined in Theorem (\ref{theorem:cauchy}), is $\sa{\BunceDeddens{\sigma}}$. This requires us to generalize techniques from \cite{Aguilar18} to our current setting.

\section{The Propinquity for order-unit-based quantum metric spaces  and Rieffel's quantum Gromov-Hausdorff distance}\label{s:iou}

The proof of completeness of the propinquity \cite{Latremoliere13c} and the construction of the inductive limit of an inductive sequence of C*-algebras share some obvious patterns, which were first exploited in \cite{Aguilar18}. In order to extend the techniques in \cite{Aguilar18} to the setting of Rieffel's distance, we first derive a new expression for $\dist_q$ motivated by the construction of the propinquity which we now follow, but in this much more relaxed framework of {\ouqcms s}.

\begin{definition}
  A \emph{quantum order-unit isometry} $\pi : (\A,\Lip_\A) \rightarrow (\B,\Lip_\B)$ between two {\ouqcms s} $(\A,\Lip_\A)$ and $(\B,\Lip_\B)$ is a positive linear map $\pi : \A\rightarrow\B$ which maps the order unit of $\A$ to the unity order of $\B$, such that:
  \begin{equation*}
    \forall b \in \B \quad \Lip_\B(b) = \inf\left\{ \Lip_\A(a) : \pi(a) = b \right\}\text{.}
  \end{equation*}
\end{definition}

\begin{definition}\label{def:ou-tunnel}
  If $(\A,\Lip_\A)$ and $(\B,\Lip_\B)$ are two {\ouqcms s}, then an \emph{order unit tunnel} $\tau = (\D,\Lip_\D,\pi_\A.\pi_\B)$ is an ordered quadruple such that $(\D,\Lip_\D)$ is an {\ouqcms}, while $\pi_\A$ and $\pi_\B$ are quantum order unit isometries from $(\D,\Lip_\D)$ onto, respectively, $(\A,\Lip_\A)$ and $(\B,\Lip_\B)$.
\end{definition}

\begin{definition}
  The \emph{extent} $\tunnelextent{\tau}$ of an order unit tunnel $\tau$ from $(\A_1,\Lip_1)$ to $(\A_2,\Lip_2)$ is:
  \begin{equation*}
    \max_{j\in\{1,2\}} \Haus{\Kantorovich{\Lip_\D}}\left(\StateSpace(\D), \left\{\varphi\circ\pi_j : \varphi\in\StateSpace(\A_j) \right\} \right) \text{.}
  \end{equation*}
\end{definition}

We remark that if $\Lip$ is an admissible Lip-norm for $(\A,\Lip_\A)$ and $(\B,\Lip_\B)$ then we can form the tunnel:
\begin{equation*}
  \left(\A\oplus\B,\Lip,(a,b)\in\A\oplus\B\mapsto a, (a,b)\in\A\oplus\B\mapsto b\right)\text{.}
\end{equation*}

The following lemma reconciles the extent of this tunnel with Rieffel's computation of $\dist_q$.

\begin{lemma}\label{extent-is-length-lemma}
  If $(\A_1,\Lip_1)$ and $(\A_2,\Lip_2)$ are two {\ouqcms s}, and if $\tau = (\A_1\oplus\A_2,\Lip,\pi_1,\pi_2)$ is a tunnel from $(\A_1,\Lip_1)$ to $(\A_2,\Lip_2)$ with $\pi_j : (a_1,a_2) \in \A_1\oplus\A_2 \mapsto a_j$ for both $j\in\{1,2\}$, then:
  \begin{equation*}
    \tunnelextent{\tau} = \Haus{\Kantorovich{\Lip}}((\A_1,\Lip_1),(\A_2,\Lip_2)) \text{.}
  \end{equation*}
\end{lemma}

\begin{proof}
  Write $\lambda = \Haus{\Kantorovich{\Lip}}((\A_1,\Lip_1),(\A_2,\Lip_2))$.

  Let $\varphi \in \StateSpace(\A_1\oplus\A_2)$. There exists $t_1 \in [0,1]$, $\varphi_1 \in \StateSpace(\A_1)$, and $\varphi_2 \in \StateSpace(\A_2)$ such that $\varphi = t_1 \varphi_1 + (1-t_1)\varphi_2$. Now, there exists $\psi_1 \in \StateSpace(\A_1)$ such that $\Kantorovich{\Lip}(\psi_1,\varphi_2)  \leq \lambda$ by definition of $\lambda$. Set $\theta = t_1\varphi_1+(1-t_1)\psi_1$. We then compute:
  \begin{align*}
    \Kantorovich{\Lip}(\varphi,\theta) 
    &= \sup\left\{ |\varphi(a,b) - \theta(a)| : \Lip(a,b) \leq 1 \right\} \\
    &\leq (1-t_1) \sup\left\{ |\varphi_2(b) - \psi_1(a)| : \Lip(a,b) \leq 1 \right\} \\
    &\leq \Kantorovich{\Lip}(\psi_1,\varphi_2) \leq \lambda \text{.} 
  \end{align*}
  By symmetry in $\A_1$ and $\A_2$, we conclude that $\tunnelextent{\tau} \leq \lambda$.

  On the other hand, let $\varphi \in \StateSpace(\A_1)$. Of course, with the usual identification, $\varphi \in \StateSpace(\A_1\oplus\A_2)$ (with $\varphi(a_1,a_2)=\varphi(a_1)$ for all $(a_1,a_2)\in\A_1\oplus\A_2$). By definition of $\tunnelextent{\tau}$, there exists $\psi\in\StateSpace(\A_2)$ such that $\Kantorovich{\Lip}(\varphi,\psi) \leq \tunnelextent{\tau}$. As $\varphi$ is arbitrary and by symmetry in $\A_1$ and $\A_2$, we conclude $\lambda\leq \tunnelextent{\tau}$.
\end{proof}

Thus, we obtain a new expression for Rieffel's distance in the spirit of the propinquity.

\begin{theorem}\label{distq=ou}
  If $(\A,\Lip_\A)$ and $(\B,\Lip_\B)$ are two {\ouqcms s}, then:
  \begin{multline*}
    \dist_q((\A,\Lip_\A),(\B,\Lip_\B)) = \inf\{\tunnelextent{\tau} : \\ \text{ $\tau$ is an order unit tunnel from $(\A,\Lip_\A)$ to $(\B,\Lip_\B)$} \}\text{.}
  \end{multline*}
\end{theorem}

\begin{proof}

  Let $\tau = (\D,\Lip,\pi_\A,\pi_\B)$ be an order unit tunnel. Let $\varepsilon > 0$.  For all $d_1,d_2 \in \D$, we define:
  \begin{equation*}
    \Lip'(d_1,d_2) = \max\left\{ \Lip(d_1), \Lip(d_2), \frac{1}{\varepsilon} \norm{d_1 - d_2}{\D}  \right\} \text{.}
  \end{equation*}
  If $\Lip'(d_1,d_2) = 0$ for some $d_1,d_2 \in\D$, then $d_1 = d_2$, $\Lip(d_1) = \Lip(d_2) = 0$ and thus $d_1 = d_2 \in \R\unit_\D$.

  Let $\varphi \in \StateSpace(\D)$. By construction:
  \begin{multline}\label{eq-1}
    \left\{ (d_1,d_2) \in \D\oplus\D : \Lip'(d_1,d_2) \leq 1, \varphi(d_1) = 0 \right\} \\
    \subseteq \left\{ d \in\D : \Lip(d)\leq 1,\varphi(d) = 0 \right\}\times \left\{ d \in\D : \Lip(d)\leq 1, |\varphi(d)| \leq \varepsilon \right\}\text{.}
  \end{multline}
  Both factors in the Cartesian product on the right hand-side of Expression (\ref{eq-1}) are compact since $\Lip$ is a Lip-norm, so  the left hand side is a subset of a compact set in $\D\oplus\D$. Since $\Lip'$, as the maximum of lower semi-continuous functions, is lower-semicontinuous (and since $\varphi$ is continuous), the set:
  \begin{equation*}
    \left\{ (d_1,d_2) \in \D\oplus\D : \Lip'(d_1,d_2) \leq 1, \varphi(d_1) = 0 \right\}
  \end{equation*}
  is closed, and thus it is compact as well.

  We thus have shown that $\Lip'$ is a Lip-norm using \cite[Proposition 1.3]{Ozawa05}.

  For all $(a,b) \in \A\oplus\B$, we set:
  \begin{equation*}
    \Lip''(a,b) = \inf\left\{ \Lip'(d,d') : \pi_\A(d) = a, \pi_\B(d') = b \right\}\text{.}
  \end{equation*}

  By \cite{Rieffel00}, the seminorm $\Lip''$ --- the quotient of $\Lip'$ for the map $(d,d') \in \D\oplus\D \mapsto (\pi_\A(d),\pi_\B(d')) \in \A\oplus\B$ --- is a Lip-norm on $\A\oplus\B$.

  Now, let $a \in \A$ with $\Lip_\A(a)\leq 1$. Since $\tau$ is a tunnel, there exists $d \in \D$ with $\Lip(d)\leq 1$ and $\pi_\A(d) = a$. Let $b = \pi_\B(d)$. As $\pi_\B$ is 1-Lipschitz, we have $\Lip_\B(b) \leq 1$.

  Thus the canonical surjection $(a,b) \in \A\oplus\B\mapsto a \in \A$ is  a quantum isometry from $(\A\oplus \B,\Lip'')$ onto $(\A,\Lip_\A)$. Similarly, $(a,b) \in \A\oplus\B \mapsto b \in \B$ is  also a quantum isometry.

  Let now $\varphi \in \StateSpace(\A)$. As $\tau$ is a tunnel, there exists $\psi \in \StateSpace(\B)$ such that $\Kantorovich{\Lip}(\varphi,\psi) \leq \tunnelextent{\tau}$. Let $(a,b) \in \A\oplus\B$ with $\Lip''(a,b)\leq 1$. By definition of $\Lip''$, there exists $d_1,d_2 \in \D$ such that $\pi_\A(d_1) = a$ and $\pi_\B(d_2) = b$, with $\Lip'(d_1,d_2) \leq 1$. We then estimate:
  \begin{align*}
    |\varphi(a) - \psi(b)| 
    &\leq |\varphi(\pi_\A(d_1)) - \psi(\pi_\B(d_2))| \\
    &\leq |\varphi(\pi_\A(d_1)) - \psi(\pi_\B(d_1))| + |\psi(\pi_\B(d_1)) - \psi(\pi_\B(d_2))| \\
    &\leq \Kantorovich{\Lip}(\varphi\circ\pi_\A,\psi\circ\pi_\B) + |\psi(\pi_\B(d_1)) - \psi(\pi_\B(d_2))| \\
    &\leq \tunnelextent{\tau} + \norm{d_1 - d_2}{\D} \\
    &\leq \tunnelextent{\tau} + \varepsilon \text{.}
  \end{align*}

  Consequently, $\Kantorovich{\Lip''}(\varphi,\psi) \leq \tunnelextent{\tau} + \varepsilon$. By symmetry, we conclude:
  \begin{equation*}
    \Haus{\Kantorovich{\Lip''}} (\StateSpace(\A),\StateSpace(\B)) \leq \tunnelextent{\tau} + \varepsilon\text{.}
  \end{equation*}

  Thus, by \cite{Rieffel00}, we conclude:
  \begin{equation*}
    \dist_q((\A,\Lip_\A),(\B,\Lip_\B)) \leq \tunnelextent{\tau} + \varepsilon
  \end{equation*}
  and since $\varepsilon > 0$ is arbitrary, we conclude $\dist_q((\A,\Lip_\A),(\B,\Lip_\B)) \leq \tunnelextent{\tau}$.

  Since $\tau$ is an arbitrary tunnel between $(\A,\Lip_\A)$ and $(\B,\Lip_\B)$, we conclude: 
  \begin{multline*}
    \dist_q((\A,\Lip_\A),(\B,\Lip_\B)) \leq \inf\{\tunnelextent{\tau} : \\ \text{ $\tau$ is an order unit tunnel from $(\A,\Lip_\A)$ to $(\B,\Lip_\B)$} \} \text{.}
  \end{multline*}

  On the other hand, if $\Lip'$ is an admissible Lip-norm on $\A\oplus\B$ then $(\A\oplus\B,\Lip',\pi_\A,\pi_\B)$, with $\pi_\A : (a,b)\mapsto a$ and $\pi_\B: (a,b)\mapsto b$, is a tunnel from $(\A,\Lip_\A)$ to $(\B,\Lip_\B)$. By Lemma (\ref{extent-is-length-lemma}), we then have:
  $\Haus{\Kantorovich{\Lip'}}(\StateSpace(\A),\StateSpace(\B)) = \tunnelextent{\tau}$.

  This completes our theorem.
\end{proof}

Remarkably if $\pi : (\A,\Lip_\A)\rightarrow(\B,\Lip_\B)$ is a quantum isometry between two {\ouqcms s} $(\A,\Lip_\A)$ and $(\B,\Lip_\B)$, it need not be a quotient map, in the following sense:
\begin{definition}\label{d:q-map}
  A surjection $\pi : \A \twoheadrightarrow \B$ between two normed vector spaces $\A$ and $\B$ is a \emph{quotient map} when:
  \begin{equation*}
    \forall b \in \B \quad \norm{b}{\B} = \inf\left\{ \norm{a}{\A} : \pi(a)=b \right\}\text{.}
  \end{equation*}
\end{definition}
Thus, it may be natural to require that quantum isometries are also quotient maps. In \cite{Rieffel00}, this matter is noted but seems inconsequential. In particular, we note that if $\A_1$ and $\A_2$ are two order unit spaces, then the maps $(a_1,a_2) \in \A_1\oplus\A_2\mapsto a_j$, for $j=1,2$, are in fact quotient maps, and also that an order-isomorphism between order unit maps is automatically a quotient map. Thus, the quantum isometries which play any role in \cite{Rieffel00}, including in the definition of $\dist_q$, are all already quotient maps. This issue also does not arise in \cite{Latremoliere13b} and subsequent work since *-epimorphisms are always quotient maps as well. However, in general, Definition (\ref{def:ou-tunnel}) allows for tunnels constructed out of maps which may not be quotient maps. But it is immediate that, if we restrict ourselves to tunnels constructed with quantum isometries which are also quotient maps, then Theorem (\ref{distq=ou}) still holds, as the only point of note is that Lemma (\ref{extent-is-length-lemma}) involves quantum isometries which are quotient maps (the rest of the argument follows unchanged).

Theorem (\ref{distq=ou}) suggests other techniques used in the theory of the Gromov-Hausdorff propinquity may be applied to Rieffel's distance. We will indeed follow this idea and provide an alternate proof of completeness for Rieffel's distance along the lines of the propinquity's proof of completeness. The reason for doing so is that the construction in the propinquity's proof are well-behaved with respect to C*-algebras, which will be helpful in our current context. Indeed, the limit is not just defined as an order unit space of continuous affine functions over some compact convex sets, but as a quotient of an order unit space.

However, it is important to stress that Theorem (\ref{distq=ou}) does \emph{not} state that the propinquity is the restriction of $\dist_q$ to C*-algebra-based quantum compact metric spaces. In fact, the proof of Theorem (\ref{distq=ou}) involves taking a quotient of a Lip-norm, which would create difficulties when working with quasi-Leibniz seminorms, which is part of the basic framework of the propinquity. In fact, $\dist_q$ and the propinquity have different coincidence properties and are different metrics, even when restricted to C*-algebras with (quasi)-Leibniz Lip-norms. Instead, Theorem (\ref{distq=ou}) states that the efforts put in deriving new techniques for the propinquity are indeed worthwhile, since removing the constraints of working only with quasi-Leibniz Lip-norms on C*-algebras (including those involved in the definition of tunnels!) simply lead us back to Rieffel's distance.

In any case, in order to use the results of \cite{Aguilar18}, we turn to the interesting exercise to adapt the proof of completeness of the propinquity from \cite{Latremoliere13b} to $\dist_q$. This comes with some interesting subtleties. We proceed our result with a standard definition and a well-known description of order unit spaces.

\begin{notation}
  If $Z\subseteq A$ is a closed convex subset of a topological $\R$-vector space $A$, then the vector space of all the continuous affine functions from $Z$ to $\R$ is denoted by $\mathcal{AF}(Z)$, where $\varphi : Z\rightarrow \R$ is affine when for all $a,b \in Z$ and $t \in [0,1]$, we have $\varphi(ta + (1-t)b) = t\varphi(a) + (1-t)\varphi(b)$.
\end{notation}
We recall a classical result due to Kadison that provides a functional representation of complete order unit spaces.
\begin{theorem}[{\cite[Theorem II.1.8]{Alfsen71}}]
  If $\A$ is a complete order unit space, and if for all $a\in \A$ we set $\widehat{a} : \varphi \in \StateSpace(\A) \mapsto \varphi(a)$, then the map $a\in\A\mapsto \mathcal{AF}(\StateSpace(\A))$ is an order isomorphism. 
\end{theorem}

\begin{theorem}\label{theorem:distq-is-complete}
  Let $(\A_n,\Lip_n)_{n\in\N}$ be a sequence of {\ouqcms s} such that for all $n\in\N$, there exists an order unit tunnel $\tau_n = (\D_n,\Lip^n,\pi_n,\rho_n)$ from $(\A_n,\Lip_n)$ to $(\A_{n+1},\Lip_{n+1})$ with $\sum_{n=0}^\infty\tunnelextent{\tau_n} < \infty$.

  Let:
  \begin{equation*}
    \B = \left\{ (d_n)_{n\in\N} \in \prod_{n\in\N} \D_n : \sup_{n\in\N} \norm{d_n}{\D_n} < \infty \right\}
  \end{equation*}
 endowed with $\norm{(d_n)_{n\in\N}}{\B} = \sup_{n\in\N}\norm{d_n}{\D_n}$ for all $(d_n)_{n\in\N}\in\B$.

Now, let:
\begin{equation*}
  \alg{K} = \left\{ (d_n)_{n\in\N} \in \B : \forall n\in \N \quad \rho_n(d_n) = \pi_{n+1}(d_{n+1}) \right\}
\end{equation*}
and 
\begin{equation*}
  \alg{L} =\{ (d_n)_{n\in\N} \in \alg{K} : \sup_{n\in\N} \Lip^n(d_n) <\infty \}. 
\end{equation*}

Let $\alg{E}$ be the closure of $ \alg{L}$ for $\norm{\cdot}{\B}$.

Let $\alg{J} = \left\{ (d_n)_{n\in\N}\in\alg{E} : \lim_{n\rightarrow\infty} \norm{d_n}{\D_n} = 0 \right\}$. 

The space $\alg{E}$ is a complete order unit space, and $\alg{J}$ is an order ideal of $\alg{E}$. There exists a weak* compact convex subset $Z$ of $\StateSpace(\alg{E})$ such that $\alg{J} = Z^\perp$ and $\lim_{n\rightarrow\infty} \Haus{\mathsf{w}}(Z,\StateSpace(\D_n))= 0$ where $\mathsf{w}$ is any metric which induces the weak* topology on the state space $\StateSpace(\alg{E})$ of $\alg{E}$.

If $\alg{F} = \bigslant{\alg{E}}{\alg{J}}$, then $\alg{F}$, endowed with the quotient order, is order-isomorphic to the order unit space of $\mathcal{AF}(Z)$. 

If, moreover, $\rho_n$ is a quotient map for all $n\in\N$, then the norm induced by the quotient order on $\alg{F}$ and the quotient norm are equal.
\end{theorem}

\begin{remark}
  We emphasize that $\alg{F}$ has two possible orders: its quotient order and the order from its structure as a space of affine functions. This already endows $\alg{F}$ with two potentially distinct seminorms. Moreover, $\alg{F}$ has a norm from being the quotient of a normed vector space by a closed subspace. A priori, this norm is different from the two other order seminorms. Our result reconciles these structures under appropriate hypothesis.
\end{remark}

\begin{proof}
It is easy to check that $\B$ and $\alg{E}$ are complete order unit spaces with order unit $\unit_{\alg{E}} = (\unit_n)_{n\in\N}$ (note that $\alg{E}$ is closed in $\B$ by continuity if the maps $\pi_n$ and $\rho_n$ for all $n\in\N$).

Now, $\alg{J}$ is a closed subspace of $\alg{E}$, so $\alg{F} = \bigslant{\alg{E}}{\alg{J}}$ is a normed vector space with norm $\norm{\cdot}{\alg{F}}$. Let $q : \alg{E} \twoheadrightarrow \alg{F}$ be the canonical surjection. Moreover, if $0\leq (d_n)_{n\in\N} \leq (j_n)_{n\in\N}$ with $(j_n)_{n\in\N} \in \alg{J}$, then for all $n\in\N$, we have $0\leq d_n \leq j_n$  and thus $\norm{d_n}{\D_n}\leq\norm{j_n}{\D_n}$, from which we readily conclude that $(j_n)_{n\in\N} \in \alg{J}$. Thus $\alg{J}$ is a order ideal in $\alg{E}$. Consequently, $\alg{F}$ is also an order vector space (with the \emph{quotient order}) with an order unit by \cite{Paulsen09} --- though this order unit is not necessarily Archimedean. Nonetheless, $q$ is a positive linear map which maps the order unit of $\alg{E}$ to the order unit of $\alg{F}$. We denote the quotient order on $\alg{F}$ simply as $\leq$, and its order unit as $e = q(\unit_{\alg{E}})$. There is also a seminorm $\norm{\cdot}{\alg{F},\leq}$ induced on $\alg{F}$ by the ordered vector space with an order unit structure on $\alg{F}$.

Any $\varphi  \in \StateSpace(\D_n)$ for any $n\in\N$ defines a state of $\alg{E}$ by setting $(d_k)_{k\in\N}\in\alg{E} \mapsto \varphi(d_n)$, and we will henceforth identify  $\StateSpace(\D_n)$ with its image for  this map.

Now, by \cite{Latremoliere13b}, the sequence $(\StateSpace(\D_n))_{n\in\N}$ converges, for the Hausdorff distance induced by any metric $\mathsf{w}$ for  the weak* topology on $\StateSpace(\alg{E})$ (which exists since $\alg{E}$ is separable), to a weak* compact and convex set $Z \subseteq \StateSpace(\alg{E})$, and moreover:
\begin{equation*}
  \alg{J} = \left\{ d \in \alg{E} : \forall \varphi\in Z \quad \varphi(d) = 0 \right\} = Z^\perp \text{.}
\end{equation*}
We note in passing that the metric used in \cite{Latremoliere13b} to define $Z$ is the {\MongeKant} induced by a Lip-norm on $\alg{E}$, but actually, the topologies induced on $\StateSpace(\alg{E})$ by the Hausdorff distance for any metrization $\mathsf{w}$ of  the weak* topology all agree with the Vietoris topology on the weak* closed subsets of $\StateSpace(\alg{E})$, so the exact metric involved is not important for this part of the proof.

If $\varphi \in \StateSpace(\alg{F})$, then $\varphi\circ q \in \StateSpace(\alg{E})$ and by construction, $\varphi\circ q(\alg{J}) = \{ 0 \}$,  thus $\varphi\circ q \in Z$. Conversely, if $\varphi \in Z$ then $\varphi$ induces a state on $\alg{F}$ since $Z^\perp = \alg{J}$. These two maps are inverse to each other and allow  us to identify $Z$ with $\StateSpace(\alg{F})$.

For any $f \in \alg{F}$, and for any $\varphi \in Z$, we set $\widehat{f}(\varphi) = \varphi(f)$. The map $f\in\alg{F} \mapsto \widehat{f}$ is linear and maps  $e$ to the order unit of $\mathcal{AF}(Z)$, i.e. the constant function $1$. This map is injective: if $\widehat{f} = 0$, then, for any $d\in\alg{E}$ with $d + \alg{J} = f$, we have $d \in Z^\perp = \alg{J}$ and thus $f = 0$.  Let $F = \left\{ \widehat{f} : f \in \alg{F} \right\}$.

We now prove that this injective linear map is positive.

Let $f \in \alg{F}$ such that $f\geq 0$. There exists $d\in\alg{E}$ and $b\in\alg{J}$ such that $d\geq b$ and $d+\alg{J} = f$. Let $\varphi \in Z$. Then $\varphi(d) \geq \varphi(b) = 0$ and thus, for all $\varphi\in Z$, we have $\varphi(f) \geq 0$. So $\widehat{f} \geq 0$.

We now  check  that the inverse map of $f \in \alg{F} \mapsto \widehat{f}$ is also  a positive map from $F$ onto $\alg{F}$.

Let $f \in \alg{F}$ such that $\widehat{f} \geq 0$, i.e. for all $\varphi\in Z$, we have $\varphi(f) \geq 0$. Let $d \in \alg{E}$ such that $d+\alg{J} = f$. Now, suppose that for some $\varepsilon > 0$, for all $N\in\N$, there exists $n\geq N$ such that $d_n \leq -\varepsilon \unit_n$.

By induction, there exists a strictly increasing function $\theta:\N\rightarrow\N$ and, for each $n\in\N$, there exists $\varphi_n\in\StateSpace(\D_{\theta(n)})$ such that $\varphi_n(d_{\theta(n)}) \leq -\varepsilon$. By compactness of $\StateSpace(\alg{E})$, there exists a weak* limit $\psi \in \StateSpace(\alg{E})$ for a subsequence of $(\varphi_n)_{n\in\N}$. Now, by construction, $\psi \in Z$ since $Z$ is the Hausdorff limit of $\StateSpace(\D_{\theta(n)})$. Now, $\psi(d) = \lim_{n\rightarrow\infty} \varphi_n(d) \leq -\varepsilon$. By assumption, $\psi(d) = \psi(f) \geq 0$. This is a contradiction. Hence, for all $\varepsilon > 0$, there exists $N\in\N$ such that if $n\geq N$ then $d_n\geq -\varepsilon \unit_n$. Consequently, by an easy induction, we can find a sequence $(\varepsilon_n)_{n\in\N}$ converging to $0$ and such that for all $n\in\N$, we have $d_n \geq -\varepsilon_n \unit_n$.  Set $b=(\varepsilon_n 1_n)_{n\in\N}$: by construction, $b_n \in \alg{J}$, and $d+b \geq 0$. Therefore $f \geq 0$.

Consequently, $\alg{F}$ is  order-isomorphic to $F$. In turn, this proves that $\alg{F}$ is an order unit space, with state space $Z$, and as $\alg{F}$ is complete, the map $f \in \alg{F} \mapsto \widehat{f} \in \mathcal{AF}(Z)$ is onto as well. The quotient order seminorm $\norm{\cdot}{\alg{F},\leq}$ on $\alg{F}$ is given by $\sup_{\varphi\in Z}|\varphi(\cdot)|$ --- and it is a norm.

We now turn to the relationship between the order norm $\norm{\cdot}{\alg{F},\leq}$ and the quotient norm $\norm{\cdot}{\alg{F}}$.

First, let $f\in \alg{F}$. Let $\varepsilon > 0$. There exists $d \in \alg{E}$ such that $d+\alg{J} = f$ and $\norm{f}{\alg{F}}\leq \norm{d}{\alg{E}} \leq \norm{f}{\alg{F}} + \varepsilon$. In particular, $-\norm{d}{\alg{E}} \unit_{\alg{E}} \leq  d \leq \norm{d}{\alg{E}} \unit_{\alg{E}}$ and thus $-\norm{d}{\alg{E}} e \leq f \leq \norm{d}{\alg{E}} e$. We conclude that $\norm{f}{\alg{F},\leq}\leq \norm{f}{\alg{F}} + \varepsilon$, and thus $\norm{\cdot}{\alg{F},\leq}\leq\norm{\cdot}{\alg{F}}$ as $\varepsilon > 0$ is arbitrary.

In general, we do not have much more to say about these two norms. However, if, for all $n\in\N$, the map $\rho_n$ is a quotient map --- as is the case, for instance, when working with C*-algebras --- then more can be concluded. Assume henceforth that $\rho_n$ is a quotient map for all $n\in\N$.

Let $f \in \alg{F}$. Let $\varepsilon > 0$. Let $d\in\alg{\alg{E}}$ such that $d+\alg{J} = f$ and $\norm{d}{\alg{E}} \leq \norm{f}{\alg{F}} + \varepsilon$. Note that $\norm{f}{\alg{F}}\leq \norm{d}{\alg{E}}$. We now make an observation.

Let $N\in\N$. Since $\norm{\pi_N(d_N)}{\A_N} \leq \norm{d_N}{\D_N}$, and since $\rho_N$ is a quotient map, there exists $d'_{N-1}\in\D_{N-1}$ such that $\norm{d'_{N-1}}{\D_{N-1}} \leq \norm{\pi_N(d_N)}{\A_N} + \frac{\varepsilon}{N} \leq \norm{d_N}{\D_N} + \frac{\varepsilon}{N}$ and $\rho_{N-1}(d'_{N-1}) = \pi_N(d_N)$. Now, repeating this process, a simple induction show that we can find $d_0' \in \D_0 ,\ldots,d_{N-1}' \in \D_{N-1}$ such that $\norm{d'_j}{\D_j} \leq \norm{d_N}{\D_N} + \frac{j \varepsilon}{N} \leq \norm{d_N}{\D_N} + \varepsilon$ and $d' = (d'_0, \ldots, d'_{N-1}, d_N, d_{N+1}, \ldots) \in \alg{E}$. By construction, $d-d' \in \alg{J}$, $d' + \alg{J} = f$ and, in particular, $\norm{f}{\alg{F}} \leq \norm{d'}{\alg{E}}$.

We note in passing that, by construction:
\begin{align*}
  \norm{d'}{\alg{E}} &= \sup_{n\in\N} \norm{d'_n}{\D_n} \\
                     &= \max\left\{ \max_{j=0}^N \norm{d'_j}{\D_j}, \sup_{j>N} \norm{d_j}{\D_j} \right\}  \\
                     &\leq \max\left\{ \norm{d}{\alg{E}} + \epsilon, \norm{d}{\alg{E}} \right\}\\
  &\leq \norm{d}{\alg{E}} + \epsilon \leq \norm{f}{\alg{F}} + 2\epsilon
\end{align*}
so we actually have $\norm{f}{\alg{F}} \leq \norm{d'}{\alg{E}} \leq \norm{f}{\alg{F}} + 2\varepsilon$ --- though, for our proof, only the lower bound on the norm of $d'$ matters.

Thus, starting from $\norm{f}{\alg{F}} \leq \norm{d'}{\alg{E}}$, and by definition of $\norm{\cdot}{\alg{E}}$, there exists $n \in \N$ such that $\norm{f}{\alg{F}} - \varepsilon \leq \norm{d'_n}{\alg{E}} \leq \norm{f}{\alg{F}} + 2\varepsilon$. If $n\leq N$ then $\norm{d'_n}{\alg{E}} \leq \norm{d_N}{\D_N}+\epsilon$, and thus $\norm{d_N}{\alg{D_N}} \geq \norm{f}{\alg{F}} - 2\varepsilon$. Thus, we have shown that for all $N\in\N$, there exists $n\geq N$ such that:
\begin{equation*}
  \norm{f}{\alg{F}} - 2\varepsilon \leq \norm{d_n}{\D_n} + \leq \norm{f}{\alg{F}} + \varepsilon \text{.}
\end{equation*}

Thus we can find a sequence of states $\varphi_n$ and a strictly increasing function $g : \N \rightarrow \N$ such that $\varphi_n \in \StateSpace(\D_{g(n)})$ for all $n\in\N$, and $\left|\varphi_n(d_{g(n)})\right| \in [\norm{f}{\alg{E}}-2\epsilon,\norm{f}{\alg{E}}+\varepsilon]$. By compactness, the sequence $(\varphi_n)_{n\in\N}$ admits a weak* convergent subsequence, whose limit is denoted by $\psi$. By definition of $Z$, we then have $\psi\in Z$, and by construction, $|\psi(f)| \in  [\norm{f}{\alg{E}}-2\varepsilon,\norm{f}{\alg{E}}+\varepsilon]$. We conclude that:
\begin{equation*}
  \norm{f}{\alg{E}} \leq |\psi(f)| + 2\varepsilon \leq \sup_{\varphi\in Z}|\varphi(f)| + 2\varepsilon = \norm{f}{\alg{F},\leq} + 2\varepsilon\text{,}
\end{equation*}
where the last equality follows from our proof that the order norm is indeed the order unit norm obtained from identifying $\alg{F}$ with $\mathcal{AF}(Z)$.

As $\varepsilon > 0$ is arbitrary, we conclude that $\norm{f}{\alg{F}} \leq \norm{f}{\alg{F},\leq}$. 

We thus have shown $\norm{\cdot}{\alg{F}} = \norm{\cdot}{\alg{F},\leq}$.
\end{proof}

\begin{corollary}\label{c:distq-lim}
  Using the hypothesis of Theorem (\ref{theorem:distq-is-complete}), and setting for all $f \in \alg{F}$:
  \begin{equation*}
    S(f) = \inf\left\{ \sup_{n\in\N}\Lip^n(d_n) : q((d_n)_{n\in\N}) = f \right\}
  \end{equation*}
  then $(\alg{F},S)$, \emph{when $\alg{F}$ is endowed with the order norm $\norm{\cdot}{\alg{F},\leq}$}, is a {\ouqcms} and:
  \begin{equation*}
    \lim_{n\rightarrow\infty} \dist_q((\A_n,\Lip_n),(\alg{F},S)) = 0\text{.}
  \end{equation*}
  If for all $n\in\N$, the surjections $\rho_n$ are quotient maps, then $\norm{\cdot}{\alg{F}} = \norm{\cdot}{\alg{F},\leq}$ and thus we can identify $\alg{F}$ with $\mathcal{AF}(Z)$ with no ambiguity.
\end{corollary}

\begin{proof}
  This follows from Theorem (\ref{theorem:distq-is-complete}) and appropriate choices of techniques in \cite{Latremoliere13b}.
\end{proof}

Any Cauchy sequence for $\dist_q$ admits a subsequence which meets the hypothesis of Theorem (\ref{theorem:distq-is-complete}), and thus the limits of Cauchy sequences for $\dist_q$ is described by Theorem (\ref{theorem:distq-is-complete}).

\section{Inductive limits of Order Unit spaces and compact quantum metric spaces}\label{s:iou}

A method for placing a quantum metric on an inductive limit of C*-algebras was introduced in \cite{Aguilar18}. This method did not assume any quantum metric structure on the inductive limit, but rather assumed quantum metric structure on each term of the inductive sequence with some compatibility conditions between the quantum metrics of each consecutive term of the inductive sequence. However, this method relied heavily on quasi-Leibniz Lip-norms and the C*-algebra structure, which works well when one has such structure. In our main example of this paper, we have seen that we are not in this position in that our Lip-norms on our terms of our inductive sequence, while quasi-Leibniz, are not all satisfying some \emph{common} Leibniz property.  Therefore, our next goal is to translate the methods of \cite{Aguilar18} to the setting of order unit spaces with Lip-norms. The key to this lies in Theorem \ref{theorem:distq-is-complete}. A result such as Theorem \ref{theorem:distq-is-complete} was automatically given by the C*-algebraic structure in \cite{Aguilar18}. First, we recall some definitions and known results from \cite{Latremoliere13}.

\begin{definition}[{\cite[Definition 3.6, Lemma 3.4]{Latremoliere13}}]\label{d:bridge}
Let $\A$, $\B$ be two unital C*-algebras.  A {\em bridge} $\gamma$ from $\A$ to $\B$ is a 4-tuple $\gamma=(\D, \omega, \pi_\A, \pi_\B)$ such that
\begin{enumerate}
\item $\D$ is a unital C*-algebra and $\omega \in \D$,
\item the set $\StateSpace_1(\omega)= \{ \psi \in \StateSpace(\D) : \forall d \in \D, \psi(d)=\psi(\omega d)=\psi(d \omega)\}$ is non-empty, in which case $\omega$ is called the {\em pivot}, and 
\item $\pi_\A : \A \rightarrow \D$ and $\pi_\B: \B \rightarrow \D$ are unital *-monomorphisms. 
\end{enumerate}
\end{definition}

The next lemma produces a characterization of lengths of the types of bridges that appear in this article, which follows immediately from definition. But, first, we introduce a definition for the types of bridges that appear  in this article. 
\begin{definition}\label{d:e-bridge}
Let $\A$ be a unital C*-algebra, and let $\B \subseteq \A$ be a unital C*-subalgebra of $\A$.  We call the 4-tuple $(\A, 1_\A, \iota, \mathrm{id}_\A)$ the {\em evident bridge from $\B$ to $\A$}, where $\iota: \B \rightarrow \A$ is the inclusion mapping and $\mathrm{id}_\A: \A \rightarrow \A$ is the identity map. 
\end{definition}

\begin{lemma}[{\cite[Definition 3.17]{Latremoliere13}}]\label{l:bridge}
Let $\A, \B$ be two unital C*-algebras and let $(\sa{\A}, \Lip_\A), (\sa{\B}, \Lip_\B)$ be two {\ouqcms s}.  If a bridge $\gamma$ from $\A$ to $\B$ is of the form $\gamma=(\D, 1_\D, \pi_\A, \pi_\B)$, then the length of the bridge is
\begin{equation*}
\begin{split}
& \lambda(\gamma |\Lip_\A, \Lip_\B)=\\
&\max \left\{\begin{array}{c}
\sup_{a \in \A, \Lip_\A(a) \leq 1} \left\{ \inf_{b \in \B, \Lip_\B(b)\leq 1} \left\{ \|\pi_\A(a)-\pi_\B(b)\|_\D\right\}\right\}, \\
\ \ \ \ \sup_{b \in \B, \Lip_\B(b) \leq 1} \left\{ \inf_{a \in \A, \Lip_\A(a)\leq 1} \left\{ \|\pi_\A(a)-\pi_\B(b)\|_\D\right\}\right\}
\end{array}\right\}
\end{split}
\end{equation*}
In particular, this holds for evident bridges.
\end{lemma}
Next, we see how lengths of bridges can be used to estimate lengths of certain tunnels. We note that the length of any bridge between two compact quantum metric ou-spaces is finite (see the discussion preceding \cite[Definition 3.14]{Latremoliere13}).
\begin{theorem}[{\cite[Theorem 3.48]{Latremoliere15b}}]\label{t:bridge-tunnel}
Let $\A, \B$ be two unital C*-algebras and let $(\sa{\A}, \Lip_\A), (\sa{\B}, \Lip_\B)$ be two {\ouqcms s}.  Let $\gamma=(\D, \omega, \pi_\A, \pi_\B)$ be a bridge from $\A$ to $\B$. Fix any $r>\lambda(\gamma|\Lip_\A, \Lip_\B)$, where $\lambda(\gamma|\Lip_\A, \Lip_\B)$ is the length of the bridge $\gamma$. 

If we define for all $(a,b) \in \A \oplus \B$
\[
\Lip^r_{\gamma|\Lip_\A, \Lip_\B}(a,b)=\max \left\{ \Lip_\A(a), \Lip_\B(b), \frac{\|\pi_\A(a)\omega - \omega \pi_\B(b)\|_\D}{r}\right\}
\]
and we let $p_\A: (a,b) \in \A \oplus \B \to a \in \A$ and $p_\B : (a,b) \in \A \oplus \B \to b \in \B$ denote the canonical surjections, then $\tau= (\A \oplus \B, \Lip^r_{\gamma|\Lip_\A, \Lip_\B}, p_\A, p_\B)$ is an order unit tunnel from $(\sa{\A}, \Lip_\A)$ to  $(\sa{\B}, \Lip_\B)$ with length $\lambda(\tau) \leq r$, and  \[
   \dist_q\left(\left(\sa{\A}, \Lip_\A\right), \left(\sa{\B}, \Lip_\B\right)\right) \leq 2r.\]
\end{theorem}

\begin{proof}
This theorem follows from the methods in \cite[Theorem 3.48]{Latremoliere15b} and our Theorem (\ref{distq=ou}).
\end{proof}

This allows us to define:
\begin{definition}\label{d:bridge-tunnel}
Let $\A, \B$ be two unital C*-algebras and let $(\sa{\A}, \Lip_\A), (\sa{\B}, \Lip_\B)$ be two {\ouqcms s}.   Let $\gamma=(\D, \omega, \pi_\A, \pi_\B)$ be a bridge from $\A$ to $\B$. We call the order unit tunnel $(\A \oplus \B, \Lip^r_{\gamma|\Lip_\A, \Lip_\B}, p_\A, p_\B)$ from $(\sa{\A}, \Lip_\A)$ to  $(\sa{\B}, \Lip_\B)$ of Theorem \ref{t:bridge-tunnel} the {\em  $(r, \gamma|\Lip_\A, \Lip_\B)$-evident tunnel} associated to the bridge $\gamma$, Lip-norms $\Lip_\A, \Lip_\B$,  and $r >\lambda(\gamma|\Lip_\A, \Lip_\B).$
\end{definition}

\begin{hypothesis}\label{main-hyp}
Let $\A=\overline{\cup_{n \in \N} \A_n}^{\|\cdot\|_\A}$ be a unital C*-algebra such that  for each $n \in \N$, we have that $\A_n$ is a unital C*-subalgebra of $\A$ and $\A_n\subseteq \A_{n+1}$. For each $n \in \N$, let $\Lip_n$ be a Lip-norm on $\sa{\A_n}$, so that $(\sa{\A_n}, \Lip_n)$ is a {\ouqcms}. Assume for all $a \in \sa{\A_n}$ that $\Lip_{n+1}(a) \leq \Lip_n (a).$  Let $(\beta(n))_{n \in \N}$ be a sequence of positive real numbers such that $\sum_{n=0}^\infty \beta(n)<\infty.$ Let $\gamma_{n}=(\A_{n+1}, 1_\A, \iota_n, \mathrm{id}_{\A_{n+1}})$ be the evident bridge from $\A_n$ to $\A_{n+1}$, and assume  $\lambda(\gamma_n |\Lip_n, \Lip_{n+1})\leq \beta(n)$.  Denote the  associated  $(2\beta(n), \gamma_n | \Lip_n, \Lip_{n+1})$-evident tunnel by $\tau_n$ an denote $\A_n\oplus \A_{n+1}=\D_n$ and its Lip-norm by $\Lip^n$. 
\end{hypothesis}

We now begin listing some results that are more or less immediate from   \cite{Aguilar18} since these results are not affected by the lack of C*-algebraic structure.

\begin{proposition}\label{p:cauchy}
Given Hypothesis \ref{main-hyp}, it holds that for each $n \in \N$, we have
\[
\dist_q ((\sa{\A_n}, \Lip_n), (\sa{\A_{n+1}}, \Lip_{n+1}))\leq 4\beta(n)
\]
and therefore $((\sa{\A_n}, \Lip_n))_{n \in \N}$ is Cauchy with respect to $\dist_q$, and we denote its limit given by Theorem \ref{theorem:distq-is-complete} and Corollary \ref{c:distq-lim} by $(\F_\A, S_\A)$.  
\end{proposition}
\begin{proof}
This is the same proof as \cite[Proposition 2.6]{Aguilar18}.
\end{proof}

Now, we give a more explicit description of what the limit {\ouqcms} $(\F_\A, S_\A)$ in Proposition (\ref{p:cauchy}) looks like under Hypothesis \ref{main-hyp}.

\begin{proposition}\label{p:limit-space-desc}
Assume Hypothesis \ref{main-hyp}.  Using notation from Proposition \ref{p:limit-space-desc} and Theorem \ref{theorem:distq-is-complete}, we have that 
\begin{enumerate}
\item \begin{align*}\B&=\Bigg\{((a^n_n, a_{n+1}^n))_{n \in \N} \in \sa{\prod_{n \in \N} (\A_n \oplus \A_{n+1})}: \\
& \quad \quad \quad \quad \quad \quad \quad  \quad \quad \sup_{n \in \N}\{\max\{\|a_n^n\|_{\A}, \|a_{n+1}^n\|_{\A}< \infty \}\}\Bigg\}
\end{align*}
\item $\alg{K}=\{ ((a^n_n, a_{n+1}^n))_{n \in \N}\in \B : \forall n\in \N, a^n_{n+1}=a_{n+1}^{n+1}\}$, 
\item and for all $d=((a^n_n, a_{n+1}^n))_{n \in \N}\in \alg{B}$, if we define 
\[
S_0(d)=\sup\{\Lip^n((a^n_n, a_{n+1}^n)) : n \in \N\},
\]
then for all $d=((a^n_n, a_{n+1}^n))_{n \in \N} \in \alg{K}$, it holds that
\[
S_0(d) = \sup_{n \in \N}\left\{ \max \left\{ \Lip_{n}\left(a_n^n \right), \frac{\left\|a^n_n-a^{n+1}_{n+1}\right\|_\A}{2\beta(n)}\right\}\right\},
\]
\item and  
\[
S_\A(f)=\inf \{S_0(d) : d\in \alg{E}, q(d)=f\}
\]
for all $f \in \F_\A$,
\item and for each $n \in \N$, it holds that 
\[
\dist_q((\sa{\A_n},\Lip_n), (\F_\A, S_\A))\leq 4\sum_{j=n}^\infty \beta(j).
\]
\end{enumerate}
\end{proposition} 
\begin{proof}
This is the same proof as \cite[Proposition 2.10]{Aguilar18} and the fact that our tunnels are built using the canonical surjections associated to $\A_n\oplus \A_{n+1}$ for all $n \in \N$, which are quotient maps. 
\end{proof}
Now,  we want to show that $\sa{\A}$ is order isomorphic to $\F_\A$. This particular part came much more easily in \cite{Aguilar18} since  $\F_\A$ is a C*-algebra there and injective *-homorphisms between C*-algebras are automatically *-isomorphisms onto their image. In our current setting, it is not as simple to provide order isomorphisms.  Hence, we have to do more work.

\begin{definition}\label{d:*-mon}
Assuming Hypothesis \ref{main-hyp} and using notation from Proposition \ref{p:limit-space-desc}, 
define $\psi_0: \sa{\A_0} \rightarrow \B$ by
\begin{equation*}
\psi_0(a_0)=((a_0,a_0))_{n \in \N},
\end{equation*}
and 
for $n \in \N \setminus \{0\}$, define $\psi_n : \sa{\A_n} \rightarrow \B$ by 
\begin{equation*}
\psi_n(a_n)=((0,0), \ldots, (0,0), (0,a_n),(a_n, a_n),(a_n,a_n), \ldots ),
\end{equation*}
where $(0,a_n) \in \D_{n-1}= \sa{\A_{n-1}}\oplus \sa{\A_n}$.
\end{definition}

\begin{lemma}\label{l:*-mon}
Assuming Hypothesis \ref{main-hyp} and using notation from Proposition \ref{p:limit-space-desc}, the map $\psi_n : \sa{\A_n} \rightarrow \B$ is an    order isometry  (not necessarily unital) such that $\psi_n(\sa{\A_n}) \subseteq \alg{E}$ for all $n \in \N$.
\end{lemma}
\begin{proof}
Fix $n \in \N$. We have that $\psi_n $ is an     isometry by construction. Since the order on $\B$ is just the coordinate order induced by each $\sa{\A_n}$, we have that $\psi_n$ is a positive map and thus an order isometry. We will now show that $\psi_n\left( \{a \in \sa{\A_n} : \Lip_{n}(a) < \infty\}\right)   \subseteq  \alg{L} $ where $\alg{L}$ was defined in Theorem \ref{theorem:distq-is-complete}. Let $a \in \sa{\A_n}$ such that $\Lip_{ n}(a) < \infty$.  If $a \in \R1_\A$, then $S_0(\psi_n(a))=0< \infty$. So, assume $a \not\in \R1_\A$. By construction, $\psi_n(\sa{\A_n}) \subseteq \alg{K}$ by Proposition \ref{p:limit-space-desc}, and we thus have $S_0( \psi_n(a))= \max \{\Lip_{n}  (a), \| a\|_\A/(2\beta(n-1))\} < \infty $.  Therefore, $\psi_n\left( \{a \in \sa{\A_n} : \Lip_{n}(a) < \infty\}\right)   \subseteq  \alg{L} $, and thus
\begin{equation*}
\begin{split}
\psi_n(\sa{\A_n})& = \psi_n\left(\overline{\{a \in \sa{\A_n} : \Lip_{n}(a) < \infty\}}^{\| \cdot \|_\A}\right)\\
& \subseteq \overline{\psi_n\left( \{a \in \sa{\A_n} : \Lip_{n}(a) < \infty\}\right)}^{\| \cdot \|_\alg{E}}  \subseteq \overline{\alg{L}}^{\| \cdot \|_\alg{E}}= \alg{E}.
\end{split}
\end{equation*}
by continuity.
\end{proof}
This lemma allows us to define:
\begin{definition}\label{d:u*-mon}
Assuming Hypothesis \ref{main-hyp}, for each $n \in \N$, by Lemma \ref{l:*-mon} we may define
 $\psi^{(n)} : \sa{\A_n} \rightarrow \F_{\A}$ by 
$
\psi^{(n)} :=q \circ \psi_n
 $
where $q : \alg{E} \rightarrow \F_{\A}$ is the quotient map.
\end{definition}

\begin{lemma}\label{l:u*-mon}
Assuming Hypothesis \ref{main-hyp}, the map
$\psi^{(n)}:\sa{\A_n} \rightarrow \F_{\A}$ of Definition \ref{d:u*-mon}  is an order unit isomorphism onto its image for each $n \in \N$. Furthermore $\psi^{(n+1)}$ restricted to $\sa{\A_n}$ is  $\psi^{(n)}$  for all $n \in \N$.
\end{lemma}
\begin{proof}
Fix $n \in \N$. The map  $\psi^{(n)}$  is  linear by construction.  For unital, we note that $\psi_n(1_\A)-1_\B \in \alg{J}$, and thus $\psi^{(n)}(1_\A)=1_{\F_\A}.$  For injectivity, assume $a\in \sa{\A_n}$ and $\psi^{(n)}(a)=0$.  Thus $\psi_n(a)  \in \alg{J}$. Hence $
0 = \lim_{n \to \infty } \|a\|_\A = \|a\|_\A$ 
which implies $a=0$. Next, let $a \in \sa{\A_n}\subseteq \sa{\A_{n+1}}$. Then, again we have  $\psi_n(a)-\psi_{n+1}(a) \in \alg{J}$, and thus $\psi^{(n)}(a)=\psi^{(n+1)}(a).$

Now, we will show that $\psi^{(n)}$ and its inverse defined on $\psi^{(n)}(\sa{\A_n})$ are positive. Note by Theorem \ref{theorem:distq-is-complete}, we use the quotient order on $\F_\A$. We begin with showing that $\psi^{(n)}$ is positive. Let $a \in \sa{\A_n}$ such that $a \geq 0$. By Lemma \ref{l:*-mon}, we have that $\psi_n(a) \geq 0$ with respect to the order on $\alg{E}$, which is the same order on $\B$. Now, consider $d=((0,0))_{n \in \N} \in \alg{E}$. We have that $d \in \alg{J}$ and $\psi_n(a)+d=\psi_n(a)\geq 0$.  Thus $\psi^{(n)}(a) \geq 0$ with respect to the quotient order on $\F_\A$. Hence $\psi^{(n)}$ is positive.

Next,  we show that the inverse of  $\psi^{(n)}$ on $\psi^{(n)}(\sa{\A_n})$ is positive. Let $ a \in \sa{\A_n}$ and assume that $\psi^{(n)}(a)\geq 0 $ in the quotient order on $\F_\A$. Thus, there exists $d \in \alg{J}$ such that $\psi_n(a)+d \geq 0$.     Now $d=((a_n^n, a^n_{n+1}))_{n \in \N} \in \prod_{n \in \N} (\sa{\A_n}\oplus \sa{\A_{n+1}})$ such that $\lim_{n\to \infty}\max\{ \|a_n^n\|_{\A}, \|a^n_{n+1}\|_{\A} \}=0.$   Now, we have that $a+a^m_m \geq 0$ for all $m \geq n+1$ since $\psi_n(a)+d \geq 0$. Hence for all $m \geq n+1$, we have 
\begin{align*}
a+a^m_m \geq 0 & \implies a\geq -a_m^m\geq -\|a_m^m\|_\A1_\A.
\end{align*}
  Hence, as $\lim_{n \to \infty} \|a^n_n\|_\A=0$, we have that $a \geq -r1_\A$ for all $ r \in \R, r >0$. Since the given order on $\sa{\A_n}$ is Archimedean, we have that $a\geq 0.$
  
  Thus $\psi^{(n)}$ is an order unit isomorphism onto its image.
\end{proof} 

Finally, we are ready to build an order unit isomorphism from $\sa{\A}$ onto $\F_\A$. This follows the same process of \cite[Theorem 2.15]{Aguilar18} up to some crucial details concerning order unit spaces.

\begin{theorem}\label{t:main}
Assume Hypothesis \ref{main-hyp}. Using notation from Proposition \ref{p:cauchy}, the following hold:
\begin{enumerate} 
\item there exists an order unit isomorphism $\psi: \sa{\A} \rightarrow \F_\A$ such that the restriction of $\psi$ to $\sa{\A_n}$ is $\psi^{(n)}$   for all $n \in \N$ of Definition \ref{d:u*-mon}, and
\item if we define $\Lip^\beta_\A:= S_\A \circ \psi$, then $(\sa{\A}, \Lip^\beta_\A)$ is a  {\ouqcms} such that $\cup_{n \in \N} \dom{\Lip_{n}} \subseteq \dom{\Lip^\beta_\A}$ with for each $n \in \N$ 
\[\dist_q \left(\left(\sa{\A_n}, \Lip_{n}\right), \left(\sa{\A}, \Lip^\beta_\A \right)\right) \leq 4 \sum_{j=n}^\infty \beta(j)\]  and therefore $ \lim_{n \to \infty} \dist_q \left(\left(\sa{\A_n}, \Lip_{n}\right), \left(\sa{\A}, \Lip^\beta_\A \right)\right)=0.$
\end{enumerate}
\end{theorem}
\begin{proof}
The fact that there exists an  order unit isomorphism $\psi: \sa{\A} \rightarrow \F_\A$  such that  the restriction of $\psi$ to $\sa{\A_n}$ is $\psi^{(n)}$     for all $n \in \N$ follows from \cite[Remark 3.6 (ii)]{Mawhinney18} and Lemma \ref{l:u*-mon}. 

Next, we show $\psi(\sa{\A})=\F_\A.$ Let $a+\alg{J} \in \F_\A$. Let $\varepsilon>0$.  There exists $b=(b_n)_{n \in \N}=((b_n^n, b^n_{n+1}))_{n \in \N} \in \alg{L}$ such that $\|a+\alg{J} - b+\alg{J}\|_{\F_\A} < \varepsilon/2$ by density. Hence, there exists $r\in \R, r>0$ such that $S_0(b)\leq r$.  Also, there exists $N \in \N, N>1$ such that $2r \cdot \sum_{j=N}^\infty \beta(j) < \varepsilon/4$.  

 Define $c=((c_n^n, c^n_{n+1}))_{n \in \N} \in \B$ in the following way:  
\[(c_n^n, c^n_{n+1})=\begin{cases}(0,0) & : 0 \leq n \leq N-2\\
(0, b^{N-1}_{N})& : n=N-1 \\
(b_n^n,  b^n_{n+1})& : n \geq N.
\end{cases} 
\]Therefore $c \in \alg{L}$ and $c- b \in \alg{J}$, which implies that $c+\alg{J}= b + \alg{J} \in \F_\A$.

Now, consider $\psi_N( b^{N-1}_{N})$ and recall that $ b^{N-1}_{N}= b^{N}_{N} $ and that $ b^n_{n+1}= b^{n+1}_{n+1}$ for all $n \in \N$ by Proposition \ref{p:limit-space-desc}.  Therefore 
\[\left\|\psi_N( b^{N-1}_{N})-c\right\|_\alg{E}=\sup_{k \in \N} \left\| b^{N}_{N}- b^{N+k+1}_{N+k+1}\right\|_\A.\]
 Since $S_0( b)\leq r< \infty$, we have that $\| b_n^n-  b^{n+1}_{n+1}\|_\A \leq r \cdot 2 \beta(n)$ for all $n \in \N$ by Proposition \ref{p:limit-space-desc}. Hence for all $k \in \N$
\begin{equation*}
\left\| b^{N}_{N}- b^{N+k+1}_{N+k+1}\right\|_\A \leq 2 r \cdot \sum_{j=N}^{N+k} \beta(j) \leq  2 r \cdot \sum_{j=N}^\infty \beta(j) < \varepsilon/4.
\end{equation*}
Thus $\|\psi_N( b^{N-1}_{N})-c\|_\alg{E} \leq \varepsilon/4$.   Therefore, since $\psi(b^{N-1}_N) =\psi^{(N)}(b^{N-1}_N) )$ as $b^{N-1}_N \in \sa{\A_N}$, we gather 
\begin{equation*}
\begin{split}
\|a+\alg{J} - \psi(b^{N-1}_N) \|_{\F_\A} & \leq \|a+\alg{J} - b+\alg{J}\|_{\F_\A}+ \| \psi(b^{N-1}_N) - b+\alg{J}\|_{\F_\A} \\
& 
< \varepsilon/2 + \| \psi(b^{N-1}_N) - c+\alg{J}\|_{\F_\A}\\
& \leq \varepsilon/2+\|\psi_N(b^{N-1}_N)- c \|_\alg{E}   \leq \varepsilon/2+\varepsilon/4 <\varepsilon.
\end{split}
\end{equation*}
by definition of quotient norm. In particular, the set $\psi(\sa{\A})$ is dense in ${\F_\A}$ with respect to the quotient norm. However, as our tunnels are built from quotient maps since they are built using the canonical surjections $(a,b) \in \A_n \oplus \A_{n+1} \mapsto a$ and $(a,b) \in \A_n \oplus \A_{n+1} \mapsto b$ (see the discussion  after Definition \ref{d:q-map}), we have that the order norm and quotient norm equal by Theorem \ref{theorem:distq-is-complete}. As $\psi$ is an isometry with respect to the order norms since it is an order unit isomorphism between (Archimedean) order unit spaces (see \cite[Corollary II.1.4]{Alfsen71}) and $\sa{\A}$ is complete, it must be the case that $\psi(\sa{\A})={\F_\A}$. 

Now, assume that $a \in  \cup_{n \in \N} \dom{\Lip_{n}}$, then there exists $N \in \N, N>1$ such that $a \in \sa{\A_N}$ and  $\Lip_{N}(a)< \infty$. Thus by Proposition \ref{p:limit-space-desc}
\begin{equation*}
\begin{split}
\Lip^\psi_\A(a)&=S_\A\circ \psi(a)= S_\A\circ \psi^{(N)}(a) \leq S_0(\psi_N(a))\\
&= \max \{\Lip_{N}(a), \|a\|_\A/(2\beta(N-1))\}< \infty.
\end{split}
\end{equation*}
The remaining follows from Proposition \ref{p:limit-space-desc} and Corollary \ref{c:distq-lim}.
\end{proof}

As a corollary, we prove the previous result for the following description of inductive limits.

\begin{corollary}\label{c:ind-lim-conv}
Let $\A=\underrightarrow{\lim}\  (\A_n, \alpha_n)_{n \in \N}$ be an inductive limit of C*-algebras (see \cite[Section 6.1]{Murphy90}), where  for each $n \in \N$, $\A_n$ is a unital C*-algebra and $\alpha_n$ is a unital *-monomorphism. For each $n \in \N$, let $(\sa{\A_n}, \Lip_n)$ be {\ouqcms}. Let $(\psi(n))_{n \in \N}$ be a summable sequence of positive real numbers.

If for each $n \in \N$, the bridge $\delta_n=(\A_n,\A_{n+1},1_{\A_{n+1}}, \alpha_n, \mathrm{id}_{\A_{n+1}})$ has length $\lambda(\delta_n |\Lip_n, \Lip_{n+1})\leq \psi(n),$ then there exists a Lip-norm $\Lip^\psi_\A$  on $\sa{\A}$ such that for each $n \in \N$,
\[
\dist_q ((\sa{\A_n}, \Lip_n), (\sa{\A}, \Lip^\psi_\A))\leq 8\sum_{j=n}^\infty \psi(j),
\]
and thus 
\[
\lim_{n \to \infty}\dist_q ((\sa{\A_n}, \Lip_n), (\sa{\A}, \Lip^\psi_\A))=0.
\]
\end{corollary}
\begin{proof}
By \cite[Section 6.1]{Murphy90}, for each $n \in \N$, let $\alpha^{(n)}: \A_n \rightarrow \A$ be the canonical unital *-monomorphism associated to $\alpha_n$, where $\alpha^{(n)}(\A_n)$ is a unital C*-subalgebra of $\A$ such that $\alpha^{(n)}(\A_n)\subseteq \alpha^{(n+1)}(\A_{n+1})$ and $\A=\overline{\cup_{n \in \N} \alpha^{(n)}(\A_n)}$.  

Next, for each $n \in \N$, we have $(\sa{\alpha^{(n)}(\A_n)}, \Lip_n \circ (\alpha^{(n)})^{-1})$ is a {\ouqcms} such that
\begin{equation}\label{eq:other-space}
\mathrm{dist}_q ((\sa{\alpha^{(n)}(\A_n)}, \Lip_n \circ (\alpha^{(n)})^{-1}), (\sa{\A_n}, \Lip_n))=0.
\end{equation}

Now, consider the bridge $\gamma_n=(\alpha^{(n)}(\A_n), \alpha^{(n+1)}(\A_{n+1}), 1_\A, \iota_n, \mathrm{id}_{\alpha^{(n+1)}(\A_{n+1})})$.  We will show that its length $\lambda (\gamma_n| \Lip_n \circ (\alpha^{(n)}))^{-1}, \Lip_{n+1} \circ (\alpha^{(n+1)})^{-1}) \leq 2 \psi(n)$. 

Let $a \in \alpha^{(n)}(\A_n)$ such that $\Lip_n \circ (\alpha^{(n)})^{-1}(a)\leq 1$. Set $a=\alpha^{(n)}(a')$ for some $a' \in \A_n$.  Then, we have that $\Lip_n (a') \leq 1$. Thus, by Lemma \ref{l:bridge}, there exists $b' \in \A_{n+1}$ such that $\Lip_{\A_{n+1}}(b') \leq 1$ and $\|\alpha_n(a')-b'\|_{\A_{n+1}} \leq 2\psi(n).$ Now, set $b=\alpha^{(n+1)}(b')$ and note $\Lip_{n+1} \circ (\alpha^{(n+1)})^{-1}(b)\leq 1.$   Next, by \cite[Section 6.1]{Murphy90}, we have 
\[
\begin{split}
\|a-b\|_\A&=\|\alpha^{(n)}(a')-\alpha^{(n+1)}(b')\|_\A=\|\alpha^{(n+1)}(\alpha_n(a'))-\alpha^{(n+1)}(b')\|_\A\\
&=\|\alpha_n(a')-b'\|_{\A_{n+1}}\leq 2\psi(n).
\end{split}
\]
The argument is symmetric if one begins with the space $(\alpha^{(n+1)}(\A_{n+1}), \Lip_{n+1} \circ (\alpha^{(n+1)})^{-1})$.  Thus, by Lemma \ref{l:bridge}, it holds that $\lambda (\gamma_n| \Lip_n \circ (\alpha^{(n)}))^{-1}, \Lip_{n+1} \circ (\alpha^{(n+1)})^{-1}) \leq 2 \psi(n)$. Therefore, the proof is complete by Theorem \ref{t:main} and Expression \eqref{eq:other-space}, and we denote $\Lip^{2\psi}_\A$ by $\Lip^{\psi}_\A$.
\end{proof}

Now, we may provide quantum metrics on inductive limits built from quantum metrics on the spaces of the inductive sequence without the requirement of any quasi-Leibniz rule. Of course, this comes at the loss of capturing the multiplicative structure of the C*-algebra, but this opens up many more possibilities for convergence results in Rieffel's quantum distance $\dist_q$.

\bigskip

We thus are now able to state a main result for this paper.

\begin{theorem}\label{theorem:main}
  If $\sigma \in \BaireSpace$, and if for all $m\in\N$ we set $\mathsf{S}_{\sigma,0} = \Lip_{\sigma,0}$ on $\sa{\CP{\sigma}{0}}$ and for all $m\in\N\setminus\{0\}$:
  \begin{multline*} 
    \forall a \in \sa{\CP{\sigma}{m}}  \quad \mathsf{S}_{\sigma,m}(a) = \\
    \max\left\{ \varkappa_m \Lip_{\sigma,m}(a) , \mathsf{S}_{\sigma,m-1}\circ\alpha_{\sigma,m}^{-1}\circ\CondExp{\sigma,m}{a} , \frac{1}{2^m}\norm{a - \CondExp{\sigma,m}{a}}{\CP{\sigma}{m}}  \right\}
  \end{multline*}
  where for all $m\in \N\setminus\{0\}$, we have $k_m= \frac{1 + 2 l_{|\cdot|}^{\boxtimes\sigma_{m-1}}(U_{\sigma,m-1})}{\sigma_m}$, and:
  \begin{equation*}
    \forall n\in \N \quad \varkappa_n = \begin{cases}
      1 \text{ if $n \in \{0, 1\}$,}\\
      \frac{\varkappa_{n-1}}{k_n} \text{ otherwise}
    \end{cases}
  \end{equation*}
  then $(\CP{\sigma}{m},\mathsf{S}_{\sigma,m})$ is a {\ouqcms} and there exists a Lip-norm $\mathsf{S}_\sigma$ on the Bunce-Deddens algebra $\BunceDeddens{\sigma}$ such that:
  \begin{equation*}
    \lim_{m\rightarrow\infty} \dist_q((\sa{\BunceDeddens{\sigma}},\mathsf{S}_\sigma), (\sa{\CP{\sigma}{m}},\mathsf{S}_{\sigma,m})) = 0 \text{.}
  \end{equation*}
\end{theorem}

\begin{proof}
  We apply Corollary (\ref{c:ind-lim-conv}) to the conclusions of Theorem (\ref{theorem:cauchy}) and the proofs of Lemma \ref{lemma:cond-lip} and Theorem \ref{theorem:bilip}, where the length of the bridges in Corollary (\ref{c:ind-lim-conv})  are calculated.
\end{proof}

We conclude with a consequence of our construction: the map which sends an element of the Baire space to its Bunce-Deddens algebra is continuous for $\dist_q$.

\begin{theorem}\label{t:main-conv}
Using the notations of Theorem (\ref{theorem:main}), the map 
\[
\beta\in (\mathcal{N}, \mathsf{d}_\mathcal{N}) \longmapsto (\sa{\B\D(\beta)}, \Lip^{\psi_\beta}_{\B\D(\beta)}) \in (\mathrm{CQMS}_\mathrm{ou}, \dist_q).
\]
is Lipschitz with Lipschitz constant at most $32$.
\end{theorem}

\begin{proof}
  Let $\beta, \eta \in \mathcal{N}$ such that $\beta \neq \eta$.  Set $n =\min \{k \in \N : \beta(k)\neq \eta(k)\}$ (so $d(\eta,\beta)=2^{-n}$).   Hence $\boxtimes\beta(k)=\boxtimes\eta(k)$ for all $k \in \{0, \ldots, n\}$. By induction, we note that $(\CP{\beta}{m},\mathsf{S}_{\beta,m}) = (\CP{\eta}{m},\mathsf{S}_{\eta,m})$ for all $m\leq n$.

  Thus, by the triangle inequality and Theorem \ref{theorem:main}, it holds that
\[
\begin{split}
\dist_q ((\sa{\B\D(\beta)}, \Lip^{\psi_\beta}_{\B\D(\beta)}), (\sa{\B\D(\eta)}, \Lip^{\psi_\eta}_{\B\D(\eta)}))
& \leq 8 \sum_{j=n}^\infty 2^{-j}+0 +8 \sum_{j=n}^\infty 2^{-j}\\
& = 8 \cdot 2^{1-n}+8 \cdot 2^{1-n}\\
& =32\cdot 2^{-n}\\
& =32\cdot \mathsf{d}_\mathcal{N}(\beta,\eta),
\end{split}
\]
which completes the proof.
\end{proof}

\bibliographystyle{amsplain}
\bibliography{thesis}
\vfill

\end{document}